\documentclass[11pt, reqno]{amsart}
\usepackage{enumerate}
  \usepackage{geometry }
 \usepackage{amsmath}
 \usepackage{amsfonts}
 \usepackage{amssymb}
  \usepackage[utf8]{inputenc} 
  \usepackage{enumerate}
 
\usepackage{amssymb,amsmath,amscd}
\usepackage{amsthm}

\usepackage{color}
 \usepackage{mathtools}
 \usepackage{enumerate}

 \numberwithin{equation}{section}

   \newtheorem{thmalph}{Theorem}

 \newtheorem{lemmab}[thmalph]{Lemma}

\newtheorem{theorem}{Theorem}[section]
 
\newtheorem*{theorem*}{Theorem }
\newtheorem{proposition}[theorem]{Proposition}%[section]

\newtheorem{lemma}[theorem]{Lemma}%[section]
\DeclareMathOperator{\tr}{tr}

 \usepackage{mathtools}
\usepackage[bbgreekl]{mathbbol}
\usepackage{amsfonts}

\DeclareSymbolFontAlphabet{\mathbb}{AMSb}
 \usepackage[
 pagebackref,  
 colorlinks=true,  citecolor=cyan,   
 citebordercolor={0 .255 .255} , 
  linkbordercolor={0 .255 .255},  
  linkcolor=cyan
 ]{hyperref}

   \usepackage{cite}

 \def\equationautorefname~#1\null{(#1)\null}
 
\usepackage{multirow}
  \usepackage{graphicx}
 \usepackage{caption}
 \usepackage{booktabs}
 \usepackage{threeparttable}
\newcommand{\C}{\mathbb C}
\newcommand{\R}{\mathbb R}

\renewcommand{\tr}{\mathbf{tr}}

\newcommand*\diff{\mathop{}\!\mathrm{d}}

      \DeclareMathOperator{\SO}{SO}
            
      \DeclareMathOperator{\cl}{C\ell}  %real Clifford
\DeclareMathOperator{\Spin}{Spin}

\usepackage{slashed}
\newcommand{\fsl}[1]{{\slashed{#1}}} % For Dirac op, spinors disctributions

      \usepackage[normalem]{ulem}
      \usepackage{color}
        \usepackage{xcolor}
      \usepackage{cancel}

 \newcommand{\hb}[1]{\textcolor{blue}{#1}}

\newcommand\blfootnote[1]{%
  \begingroup
  \renewcommand\thefootnote{}\footnote{#1}%
  \addtocounter{footnote}{-1}%
  \endgroup
}
\begin{document}

\title[Strichartz's  conjecture for spinors]
{
%\kk{An $L^2$-range characterization of the generalized spectral projections related to the Dirac operator} 
Strichartz's conjecture for the  spinor bundle\\ over the real hyperbolic space}

\author{Abdelhamid Boussejra}
\address{Abdelhamid Boussejra : Department of Mathematics, Faculty of Sciences, University Ibn Tofail, Kenitra, Morocco 
}
\curraddr{}
\email{boussejra.abdelhamid@uit.ac.ma}
\thanks{}

\author{Khalid Koufany}
\address{Khalid Koufany : Université de Lorraine, CNRS, IECL, F-54000 Nancy, France}
 
\curraddr{}
\email{khalid.koufany@univ-lorraine.fr}
\thanks{}
\date{\today}
 	
 \maketitle

\begin{abstract}  %Version \today\\
\begin{sloppypar}
Let $H^n(\mathbb R)$ denote the real hyperbolic space realized as the symmetric space $\Spin_0(1,n)/\Spin(n)$. In this paper, we provide a characterization for the image of the Poisson transform for $L^2$-sections
  of the spinor bundle over the boundary $\partial H^n(\mathbb R)$. As a consequence, we obtain an $L^2$ uniform estimate for   the generalized   spectral projections associated to 
 the spinor bundle over   $H^n(\mathbb R)$, thereby extending Strichartz’s conjecture from the scalar case to the spinor setting.
 %thus generalizing Strichartz's conjecture from the scalar case.
 \end{sloppypar}
  \end{abstract}
\blfootnote{\emph{2010 AMS Classification : \rm 43A90, 43A85, 58J50, 15A66}} 

\tableofcontents
 
\section{Introduction}  
 Let $X=G / K$ be  a Riemannian symmetric space  of noncompact type. Fix an Iwasawa decomposition $G=KAN$ of $G$ and let $\mathfrak a$ be the Lie algebra of $A$. We write $g=\kappa(g){\rm e}^{H(g)}n(g)$ according  to the given Iwasawa decomposition.  
For $f \in L^2(X)$ we   consider the  generalized spectral  projection  $f_\lambda=\mathcal{Q}_\lambda f \in C^{\infty}(X)$  defined for almost every $\lambda \in \mathfrak{a}^*$  by
$$
f_\lambda(g)=\mathcal{Q}_\lambda f(g)=|{c}(\lambda)|^{-2} \int_K e^{-(i \lambda+\rho)(H(g^{-1}k)} \mathcal{F} f(\lambda, k) d k,
$$   
where $ \mathcal{F} f$ is the Helgason Fourier transform of $f$, $c(\lambda)$ the Harish-Chandra $c$-function and $\rho$ the half sum of   positive restricted roots.  
Then one may see that the function $f_\lambda$ is a joint  eigenfunction  for    the algebra $\mathbb{D}(X)$ of $G$-invariant differential operators on $X$.

In his  paper  \cite{strichartz}, Strichartz initiated the project   of reformulating harmonic
analysis in terms of  the spectral projections $\mathcal Q_\lambda$. He 
  investigated relations between properties of $f$ and $f_\lambda$ and  
    characterized the image    under the spectral projections of  the spaces $L^2(X)$, $\mathcal S(X)$, $\mathcal S^{\prime}(X)$ and $\mathcal D(X)$   when $X$ is the Euclidean space or  the real hyperbolic space.

Strichartz proposed then   the following  conjecture for any Riemannian symmetric space  of noncompact type and rank $r$ :\\

{\bf Strichartz's conjecture } (see \cite[Conjecture 4.5, Conjecture 4.6]{strichartz}) 
 Let $\lambda\in\mathfrak a^*_{\text{reg}}$. For $f_\lambda=\mathcal Q_\lambda f$ with $f \in L^2(X)$   we have
\begin{align}\label{st0}
\|f\|_2^2  \approx  \sup _{t, z} \int_{\mathfrak{a}_{+}^*} \frac{1}{t^r} \int_{B_t(z)}\left|f_\lambda(x)\right|^2 \diff x \diff \lambda
\end{align}
and for any fixed $z$,
\begin{align}\label{st1}
\|f\|_2^2=2^r \pi^{r / 2} \Gamma\left(\frac{r}{2}+1\right) \lim _{t \rightarrow \infty} \int_{\mathfrak{a}_{+}^*} \frac{1}{t^r} \int_{B_t(z)}\left|f_\lambda(x)\right|^2 \diff x \diff \lambda,
\end{align}
where $B_t(z)$ is the open ball in $X$ centered at $z$ with radius $t$  and $\Gamma$ is the classical Euler Gamma function.
 
  Conversely, if $f_\lambda$ is any family of joint eigenfunctions of $\mathbb{D}(X)$ such that    the right-hand side of \eqref{st0} or \eqref{st1} is finite, then  there exists $f\in L^2(X)$ satisfying  $f_\lambda=\mathcal Q_\lambda f$. \\

The above  characterization  of $\mathcal Q_\lambda \left(L^2(G/K)\right)$ is related to  the image characterization of $L^2(K/M)$ under the Poisson transform $\mathcal P_\lambda$, and therefore closed  to the Helgason conjecture stated in \cite{H1}.  Here $M$ is the centralizer of $A$ in $K$.

  Continuing Strichartz's   project many authors   proved a spectral projections image characterization of $L^2(G/K)$ and Poisson transform image characterization of $L^2(K/M)$. For rank one case we mention e.g.  \cite{IB, Obray,Ionescu, BS}.  
    For a  treatment of the case of any Riemannian symmetric space  of noncompact type  we refer to \cite{Kaizuka}.  A recent account for a class of vector bundles  can be found in\cite{BIO, BO,BBK2,BBK, BK-24}.

The present paper deals with   Strichartz’s conjecture on Poisson transforms and the  generalized spectral projections for  the spinor bundle  over   the real hyperbolic space $H^n(\mathbb R)$. We consider the realization of $H^n(\mathbb R)$ as  the homogeneous space $G/K$, where $G = \Spin_0(1,n)$ is the spinor conformal group and $K = \Spin(n)$ its maximal compact subgroup. 

Fix $\tau$ to be the spin representation $\tau_n$ ($n$ odd) or the half-spin representation $\tau_n^\pm$ ($n$ even). 
  Let $G\times_K V_\tau$ denote  the spinor bundle associated to $\tau$. Its $L^2$-sections space will be identified to the space $L^2(G,\tau)$ consisting of functions $f: G\to V_\tau$ which are square-integrable  with respect to the invariant measure of $G/K$ and satisfy  equivariance condition $f(gk)=\tau(k^{-1})f(g)$, for all $g\in G$ and $k\in K$.
%Let $L^2(G,\tau)$ be the space of $L^2$-sections of   the spinor bundle $G\times_K V_\tau$.
Then according to the Helgason Fourier inversion formula (see \cite{CP}), any $f\in C_c^\infty(G,\tau)$ 
can be written as
 \begin{align}\label{decomp}
 f=\sum_{\sigma\in\widehat{M}(\tau)}\int_0^\infty \mathcal Q_{\sigma,\lambda}^\tau f \diff \lambda,
 \end{align}
 where  
 $\widehat{M}(\tau)$ is the set of  unitary irreducible representations of $M=\Spin(n-1)$   that occur in the restriction $\tau_{|M}$ and  $\mathcal Q^\tau_{\sigma,\lambda}$ is   the  partial generalized spectral projection   which can be written by means of  the partial Helgason Fourier transform $\mathcal F_{\sigma,\lambda}^\tau$ (see \eqref{12-dec-F} and \eqref{proj-Q1})  and the Poisson transform $\mathcal P_{\sigma,\lambda}^\tau$ (see \eqref{Poisson}) through the relation \footnote{ Notice that the operator $\mathcal{Q}_{\sigma, \lambda}^\tau$ is well defined on $L^2(G,\tau)$ by the Plancherel Theorem.}
 $$
\mathcal{Q}_{\sigma, \lambda}^\tau f(g)=\nu_\sigma(\lambda)\left(\mathcal{P}_{\sigma, \lambda}^\tau\left(\mathcal{F}_{\sigma, \lambda}^\tau f(\cdot)\right)(g).\right.
$$
Here $\nu_\sigma(\lambda)$ is the Plancherel density.

 Let $\mathcal E_{\sigma,\lambda}(G,\tau)$ denote the space of joint eigenfunctions of the of algebra of $G$-invariant differential operators acting on the bundle $G\times_K V_\tau$, with eigenvalue $\lambda^2$ if $n$ is even and $\pm\lambda$ if $n$ is odd, (see \eqref{g1}, \eqref{g2}).
 Then it is known that the image of the Poisson transform belongs to   the joint eigenspace $\mathcal{E}_{\sigma,\lambda}(G,\tau)$, see e.g. \cite{CP}. 
Consequently, the decomposition \eqref{decomp} may be seen as a spectral decomposition of $L^2$-sections $f$ into the family of joint eigensections $\mathcal Q_{\sigma,\lambda}^\tau f$. 
 
Our first main theorem is a generalization of Strichartz's conjecture   for the generalized spectral projections  $Q_\lambda^\tau =
(\mathcal Q_{\sigma,\lambda}^\tau)_{\sigma\in\widehat{M}(\tau)}: L^2(G,\tau)\to \sum_{\sigma\in\widehat{M}(\tau)} \mathcal{E}_{\sigma,\lambda}(G,\tau)$, see Theorem \ref{main-th-proj}. %More precisely %\sout{In particular} we obtain  a characterization of $\mathcal{Q}^\tau\left(L^2(G,\tau)\right)$ \sout{the image of $L^2(G,\tau)$ by the spectral projections $(\mathcal{Q}^\tau_\lambda)_{\lambda\in (0,\infty)}$.}

To prove Theorem \ref{main-th-proj}, we will begin by examining a uniform $L^2$-boundedness %and characterizing the image of $L^2(K,\sigma)$ under the Poisson transform $\mathcal P_{\sigma,\lambda}^\tau$, %\sout{we will first study a uniform $L^2$-boundedness and the image characterization of $L^2(K,\sigma)$ by the Poisson transform,} 
and characterize the joint eigensections in $\mathcal{E}_{\sigma,\lambda}(G,\tau)$ which are Poisson transform of $L^2$-sections of the spinor bundle $K\times_M V_\sigma$.
 This objective is addressed in the second main theorem of this paper, see Theorem \ref{main-th-Poisson}, which focuses on Strichartz's conjecture for the Poisson transform. %\sout{This is the purpose of the second main theorem of this paper on Strichartz's conjecture for the Poisson transform, see Theorem \ref{main-th-Poisson}. }

 For $\lambda\in\mathbb C$ with $\Re(i\lambda)>0$, a characterization of $\mathcal P_{\sigma,\lambda}^\tau\left(L^2(K,\sigma)\right)$ has been obtained in \cite{BBK2} (see also \cite{BBK} for differential forms)  using essentially a Fatou type theorem,
\begin{align*}
\lim_{t\rightarrow \infty}{\rm e}^{(\rho-i\lambda)t}\mathcal{P}_{\sigma,\lambda}^\tau\, f(ka_t)= \frac{\dim\tau}{\dim\sigma}  \mathbf c(\lambda,\tau) f(k),
\end{align*}
where $\mathbf c(\lambda,\tau)$ is the spinor-valued Harish-Chandra $c$-function.
 Notice that the above asymptotic is not valid for real parameter $\lambda$, since  $\mathcal{P}_{\sigma,\lambda}^\tau\, f(ka_t)$ has  oscillating terms at infinity.

 To overcome this problem for real parameters $\lambda$, we will  establish   a scattering formula   for   Poisson integrals, see Theorem \ref{asymp-Poisson}.  This idea goes back to \cite{Kaizuka}.  \\

  %%%%%%%%%%%%
 We now give a summary of the results in this paper. 
 Following the general preliminaries in Section 2, we introduce the Helgason Fourier transform on the spinor bundle in Section 3 and establish a Helgason Fourier restriction theorem.
  As a consequence, we obtain in section 4  a uniform estimate for  the Poisson transform. Then in Section 5 we give the precise  announcement of the main theorems. Section 6   deals with a technical proof of the asymptotic expansion of $\tau$-spherical functions from which we get the asymptotic expansion of Poisson integrals. Section 7 is devoted to the proof of the main    Theorem \ref{main-th-Poisson} and Section 8 to the proof of the main    Theorem \ref{main-th-proj}. The  appendix contains  technical lemmas using a realization of the conformal spin group by Vahlen matrices.
  
 % Several proofs will be abbreviated, as similar arguments to those used in \cite{BK-24} for    the bundle of differential forms over $H^n(\mathbb R)$   can be adapted here.

\section{Notations and Preliminaries}

 Let $H^{n}(\mathbb{R})$ be the $n$-dimensional real hyperbolic space ($n \geq 2$)
  viewed as the rank one symmetric space of the noncompact type $G/K$ where $G=\Spin_0(1,n)$ and $K=\Spin(n)$ are the two-fold coverings of $\SO_0(1,n)$ and $\SO(n)$ respectively.

 Let $\mathbf {\mathfrak g}\simeq\mathfrak{so}(1,n)$ and $\mathbf {\mathfrak k}\simeq\mathfrak{so}(n)$ be the Lie algebras of $G$ and $K$ respectively  and write $\mathfrak g=\mathfrak k\oplus \mathfrak p$ for the Cartan decomposition of $\mathfrak g$.
 The tangent space $T_o(G / K) \simeq   \mathfrak{p}$ of $G / K=H^n(\mathbb{R})$ at the origin $o=e K$ will be identified with the vector space $\mathbb{R}^n$.

Let  $H_0\in\mathfrak p$ such that 
%Let
%$$
%H_0=\left(\begin{array}{ccc}
%0 & 0 & 1 \\
%0 & 0_{n-1} & 0 \\
%1 & 0 & 0
%\end{array}\right) \in \mathfrak{p}
%$$
%Then 
$\mathfrak{a}=\mathbb{R} H_0$ is a Cartan subspace in $\mathfrak{p}$. Let $A=\exp\mathfrak a=\{a_t= e^{tH_0},\; t\in\mathbb R\}$ be the corresponding analytic Lie subgroup of $G$.
Let $\alpha \in \mathfrak{a}^*$ such  that $\alpha\left(t H_0\right)=t$.  
The positive restricted root subsystem is $\Sigma^{+}(\mathfrak{g}, \mathfrak{a})=\{\alpha\}$ and the corresponding positive Weyl chamber is $\mathfrak a^+=\{H\in\mathfrak a,\; \alpha(H)>0\}\simeq (0,\infty)$.   We will identify $\mathfrak a_{\mathbb C}^*$ and $\mathbb C$ via the map $\lambda \alpha \mapsto \lambda$. Thus, the half-sum of positive roots is $\rho=(n-1)\alpha/2 = (n-1)/2$.

Let $\mathfrak{n}=\mathfrak{g}_\alpha$ be the   positive root subspace and $N=\exp(\mathfrak n)$ the corresponding abelian analytic subgroup of $G$.  The groupe $G$ has an Iwasawa decomposition $G=KAN$
and a Cartan decomposition
$G=K\overline{A^+}K$ where  $\overline{A^+}$ is the closure  in $G$ of $A^+=\exp(\mathfrak a^+)=\{a_t,\; t>0\}$.
For any $g\in G$, we denote $H(g)\in\mathfrak a$ and $A^+(g)\in\overline{\mathfrak a^+}$ the unique elements such that
$$g=k(g) e^{H(g)} n(g),\;\; k(g)\in K, n(g)\in N,$$
and
$$g=k_1(g) e^{A^+(g)} k_2(g),\; k_1(g), k_2(g)\in K.$$
 
 %Let $\diff g$ and $\diff k$ be the Haar measures on $G$ and $K$ respectively.  Using  the Cartan decomposition, the following integral formula holds for any integrable function $f$ %the Haar measure on $G$ writes
%  \begin{equation}\label{chg-cartan}
%\int_G   \diff g f(g)=\int_K \diff k_1 \int_0^{\infty} \diff t(2 \sinh t)^{n-1} \int_K \diff k_2 f\left(k_1 a_t k_2\right).
%\end{equation}

 We define the invariant measure on $G/K$ by %In the Cartan decomposition, the Haar measure on $G$ writes
  \begin{equation}\label{chg-cartan}
\int_{G/K}   \diff (gK) f(gK)  =\int_0^{\infty} \diff t(2 \sinh t)^{n-1} \int_K \diff k f\left(k a_t \right).
\end{equation}
for any integrable function $f$  on $G/K$. Here $K$ is equipped with its normalized Haar measure.\\

Let $\cl(n)=\cl(\mathbb R^n)$ denotes the Clifford algebra of $\mathfrak p\simeq \mathbb R^n$ (see the appendix for more details).
 A Clifford module $(\tau, V_\tau)$ is a complex vector space $V_\tau$   together with a (left) action $\tau$ of the Clifford algebra $\cl(n) $ on $V_\tau$  (the space of spinors associated with the complexfication of $\mathfrak p=T_o(G/K)$). By restriction, the action $\tau$ yields representations of the group   $\Spin(n)$, also denoted by $\tau$.

When $n$ is even, say $n=2 m$, there exists, up to equivalence, a unique irreducible Clifford module $(\tau_n,V_{\tau_n})$ of dimension $2^m$. As a representation of $\Spin(n)$, $\tau_n$  splits into two irreducible non-equivalent representations, the half spinors representations $(\tau_n^+,V_{\tau_n^+})$ and $(\tau_n^-,V_{\tau_n^-})$, each of dimension $2^{m-1}$.

When $n$ is odd, say $n=2 m+1$, there exist two non-equivalent irreducible Clifford modules of dimension $2^m$. As representations of the  $\Spin(n)$, they are irreducible and equivalent, thus leading to a unique spinor representation  $(\tau_n,V_{\tau_n})$.

We will hereafter simply denote
$$\Lambda=\{\tau_n (\text{$n$ odd}), \tau_n^\pm (\text{$n$ even})\}.$$

%Let $\tau\in\Lambda$, then   the spinor   bundle over $H^n(\mathbb R)$ is the homogeneous vector bundle $G \times{ }_{K} V_{\tau}$.\hb{remplacer la phrase précédente par :\\ 
 
 The spinor bundle  is defined as    the   homogeneous vector bundle $G\times_K V_\tau$ over $G/K$ corresponding to $\tau\in \Lambda$.
    We identify its smooth sections space $C^{\infty}\left(G\times_K V_\tau\right)$ with the space $C^\infty(G,\tau)$ of smooth functions $f \colon G\to V_\tau$ such that 
    \begin{equation}\label{t-equiv}
    	f(g k)=\tau\left(k^{-1}\right) f(g)\quad  \text{for all $g \in G$ and $k \in K$.}
    \end{equation} 
     Since $K$ is compact, we also identify 
    $C_{\mathrm{c}}^{\infty}\left(G\times_K V_\tau\right)$ with $C_{\mathrm{c}}^{\infty}(G ; \tau)$. Here $C^\infty_c$ means $C^\infty$ with compact support. Similarly, we denote by $L^2(G ; \tau)$ the space of $V_\tau$-valued $L^2$-functions on $G$ satisfying the transformation rule \eqref{t-equiv}.

 Let $M\simeq \Spin(n-1)$ be the centralizer of $A$ in $K$. 
  To distinguish between spinor representations of $K=\Spin(n)$ and spinor representations of $M=\Spin(n-1)$ we will 
  henceforth use the Greek letter   $\sigma$ to denote  representations of    $M$. 
  Thus the spinor representation of $M=\Spin(n-1)$ is denoted by $\sigma_{n-1}$ (when $n-1$ is odd) and the half-spinor representations are denoted by $\sigma_{n-1}^\pm$ (when $n-1$ is even).

 Recall the following branching law of $(K,M)=(\Spin(n),\Spin(n-1))$, see e.g. \cite[Chapter 8]{GW}.
 
\begin{enumerate}%[\upshape 1)]
\item[$(1)$]  If $n$ is even,   then   ${\tau_n^\pm}_{|M} = \sigma_{n-1}$.
\item[$(2)$] If $n$ is odd,   then ${\tau_n}_{|M}=\sigma^+_{n-1}\oplus \sigma^-_{n-1}$. 
  \end{enumerate}
   
For $\tau\in\Lambda$ we denote by  $\widehat{M}(\tau)$ the set of representations in $\widehat{M}$ that occur in the restriction of $\tau$ to $M$ with multiplicity one according to  above branching rule. Thus $\sigma \in \widehat{M}(\tau)$ means $\sigma=\sigma_{n-1}$ or $\sigma=\sigma^\pm_{n-1}$ in accordance with the parity of $n$.

%\hb{ je propose d'enlever ce paragraphe :
%If we fix $\tau\in\Lambda$, then for any   $\sigma\in\widehat{M}(\tau)$ we  consider the vector bundle $K\times_M V_\sigma$ over the boundary $\partial H^n(\mathbb R)=$ $K/M$. We identify  the space of smooth sections of the given bundle with the space $C^\infty(K,\sigma)$ of smooth   $F : K\to V_\sigma$ such that
%$$F(km)=\sigma(m)^{-1} F(k) \quad \text{ for all } k\in K, m\in M.$$
% We will also denote by $L^2(K,\sigma)$ (resp. $C^{-\omega}(K,\sigma)$) the space of square integrable $V_\sigma$-valued functions (resp. $V_\sigma$-valued hyperfunctions) on $K$ satisfying the above covariant relation.}\\

  Following \cite{Campo}   the vector-valued Helgason Fourier transform of $f\in  C^\infty_c(G,\tau)$ is  the $V_\tau$-valued function on $\mathfrak a^*_{\mathbb C}\times K=\mathbb C\times K$ given by
  \begin{equation*}
 \mathcal F^\tau f(\lambda,k)= \int_G \diff g\,  e^{(i\lambda -\rho)H(g^{-1}k)} \tau(\kappa(g^{-1}k))^{-1}
f(g). 	
  \end{equation*}
 
  Then for fixed $\sigma\in \widehat M(\tau)$, the   partial Helgason Fourier transform  $\mathcal F_\sigma^\tau f$ of $f\in  C^\infty_c(G,\tau)$ is  the $V_\sigma$-valued function on $\mathfrak a^*_{\mathbb C}\times K=\mathbb C\times K$ given by  
  \begin{equation}\label{12-dec-F}
\mathcal F_\sigma^\tau f(\lambda,k)= \sqrt{d_{\tau,\sigma}}\; P_\sigma\mathcal F^\tau f(\lambda,k),
\end{equation}
where %$d_{\tau,\sigma}=\frac{\dim \tau}{\dim \sigma}$ 
$$
d_{\tau,\sigma}=\frac{\dim \tau}{\dim \sigma}=\begin{cases}
 1&\text{ if $n$ is even},\\
  {2}&\text{ if $n$ is odd.}\\
 \end{cases}	
$$
and 
$P_\sigma$ is   the orthogonal  projection of $V_\tau$ onto  its $M$-isotypical component $V_\sigma$ .
We will also consider the function $\mathcal F_{\sigma,\lambda}^\tau f$ as a function on $K$ given by   $\mathcal F_{\sigma,\lambda}^\tau f(k)=\mathcal F_\sigma^\tau f(\lambda,k)$. Then $\mathcal F_{\sigma,\lambda}^\tau f$ belongs to $L^2(K,\sigma)$, the space of square integrable  $V_\sigma$-valued functions which are right $M$-covariant.

As proved in \cite[Theorem 7.3]{CP}, the Plancherel theorem for $L^2(G,\tau)$ is as follows:

$(1)$ If $n$ is even, then the  Helgason Fourier transform on $C_{\mathrm{c}}^{\infty}\left(G, \tau_n^{ \pm}\right)$ is inverted by the following formula
\begin{equation}\label{Inv-Formula-even}
f^\pm(g)= \int_0^\infty   \nu(\lambda) \diff \lambda \int_K  \diff k \, e^{-(i\lambda+\rho)H(g^{-1}k)}\tau_n^\pm (\kappa(g^{-1}k)) \mathcal F^{\tau_n^\pm}_{\sigma_{n-1},\lambda} f^\pm(k), 
\end{equation}
and we have the Plancherel formula 
\begin{equation}\label{Plan-Formula-even}
\|f^\pm \|_{L^2\left(G, \tau_n^{ \pm}\right)}^2=\int_0^{\infty}  \nu(\lambda)\diff \lambda\,   
 \| \mathcal F^{\tau_n^\pm}_{\sigma_{n-1},\lambda}  \left(f^\pm \right)\|_{L^2\left(K, \sigma_{n-1}\right)}^2.
\end{equation}
Furthermore,  $\mathcal F^{\tau_n^\pm}=\mathcal F^{\tau_n^\pm}_{\sigma_{n-1}}$ extends to a bijective isometry  
\begin{equation*}
\mathcal F^{\tau_n^\pm} \colon L^2\left(G, \tau_n^{ \pm}\right) \xrightarrow{\sim}
L^2\left(\mathbb{R}_{+} ; \nu(\lambda) \diff\lambda ; L^2\left(K, \sigma_{n-1}\right)\right).
\end{equation*}

$(2)$ If $n$ is odd, then the  Helgason Fourier transform on $C_{\mathrm{c}}^{\infty}\left(G, \tau_n\right)$ is inverted by the following formula
  \begin{equation}\label{Inv-Formula-odd}
f(g)=\sqrt{2}    \int_0^\infty     \nu(\lambda) \diff \lambda \int_K \diff k \,
 e^{-(i\lambda+\rho)H(g^{-1}k)}\tau_n (\kappa(g^{-1}k))\bigl[ 
  \mathcal F^{\tau_n}_{\sigma_{n-1}^+,\lambda}f(k)  + 
 \mathcal F^{\tau_n}_{\sigma_{n-1}^-,\lambda} f(k)\bigr],
\end{equation}
and we have the Plancherel formula 
\begin{equation}\label{Plan-Formula-odd}
\|f\|_{L^2\left(G, \tau_n\right)}^2=   \int_0^{\infty}  \nu(\lambda) \diff\lambda  
\left[\|
 \mathcal F^{\tau_n}_{\sigma_{n-1}^+,\lambda}(f)
\|_{L^2\left(K, \sigma_{n-1}^{+}\right)}^2 
 +\|
 \mathcal F^{\tau_n}_{\sigma_{n-1}^-,\lambda}(f)
 \|_{L^2\left(K, \sigma_{n-1}^{-}\right)}^2\right].
\end{equation}
Furthermore, 
the Fourier transform 
$\mathcal{F}^{\tau_n}=\bigl( \mathcal F^{\tau_n}_{\sigma_{n-1}^+}  ,  \mathcal F^{\tau_n}_{\sigma_{n-1}^-}\bigr)$ extends to a bijective isometry 
\begin{equation*}
\mathcal F^{\tau_n}  \colon L^2\left(G, \tau_n\right) \xrightarrow{\sim}
L^2\left(\mathbb{R}_{+} ; \nu(\lambda) \diff\lambda  ; L^2\left(K, \sigma_{n-1}^{+}\right)\right) \oplus L^2\left(\mathbb{R}_{+};  \nu(\lambda) \diff\lambda  ; L^2\left(K, \sigma_{n-1}^{-}\right)\right).
\end{equation*}
Above, the Plancherel density $\nu(\lambda)$ is given by
 \begin{equation*}
 	\nu(\lambda)=\begin{cases}
 		\frac{2}{\pi} |\mathbf c(\lambda)|^{-2} & \text{ if $n$ is even},\\
 		\frac{1}{\pi} |\mathbf c(\lambda)|^{-2} & \text{ if $n$ is odd},\\
 	\end{cases}
 \end{equation*}
where
\begin{equation}\label{c-function1}
\mathbf c(\lambda)=2^{n-i 2 \lambda} \frac{\Gamma(n / 2) \Gamma(i 2 \lambda)}{\Gamma(i \lambda+n / 2) \Gamma(i \lambda)}.
\end{equation}
The explicite expression of $\nu(\lambda)$ is given by (see \cite[Theorem 6.3]{CP})
\begin{equation}\label{d1}
\nu(\lambda)=
\begin{cases}
2^{3-2n}[(n/2-1)]^{-2}\lambda\, \coth(\pi\,\lambda)\prod_{j=1}^{\frac{n}{2}-1}(\lambda^2+j^2) &\text{if $n$ is even},	\\
2^{-2}\pi^{-1}[(\frac{n-1}{2})( \frac{n-1}{2}+1)(n-2)]^{-2}\prod_{j=1}^{\frac{n-1}{2}}(\lambda^2+(j-1/2)^2)& \text{if $n$ is odd}.
\end{cases}
\end{equation}
Notice  that the Plancherel density  satisfies the following estimate
 \footnote{The symbol $f \asymp g$ means that there exists positive constants $C_1$ and $C_2$ such that $C_1 g(\lambda) \leqq f(\lambda) \leq C_2 g(\lambda)$.} 
\begin{equation}\label{dasy}
\nu(\lambda)\asymp (1+\mid \lambda\mid)^{n-1} \quad \text{ for }\lambda\in\mathbb{R}.
\end{equation}

\section{The Helgason Fourier restriction theorem}

In this section, we shall investigate  the Helgason Fourier   restriction 
theorem for $L^2(K,\sigma)$.
To this end,  let us first review some facts about  the Radon transform.

    The   Radon transform of $f \in C_{\mathrm{c}}^{\infty}(G , \tau)$ is the $V_{\tau}$-valued function defined by (see e.g. \cite{BOS})
$$
\mathcal R^\tau f (g)=\mathcal R^\tau f(t,k)=e^{\rho t } \int_N f(ka_tn) \diff  n,\quad g=ka_tn\in G.
$$
 Then, for each $\sigma\in \widehat M(\tau)$ we define the   partial Radon transform of $f\in C_{\mathrm{c}}^{\infty}(G ; \tau)$ by
$$\mathcal R^\tau _\sigma f(t,k)=\sqrt{ d_{\tau,\sigma} } P_\sigma \mathcal R^\tau f(t,k).$$

For every real number $R>0$,  set $B(R):=\{gK\in G/K, d(o,gK)<R\}$, where $d(\cdot,\cdot)$ is the hyperbolic distance on $G/K=H^n(\mathbb R)$. Since (see \cite[page 476]{H})  
$$
d(eK,ka_tnK)\geq  |t|,  \quad k\in K, n\in N,
$$
we have %we easily see that, if $f\in C_c^\infty(G,\tau)$ with $\operatorname{supp}f\subset B(R)$, then  
\begin{align}\label{RS}
\operatorname{supp}  \mathcal{R}^\tau_\sigma f\subset [-R,R]\times K,
\end{align}
 for any $f\in C_c^\infty(G,\tau)$ with $\operatorname{supp}f\subset B(R)$.

Recall also the following useful relation, %Notice finally that the    partial Radon transform  is related to   the partial Helgason Fourier transform   by the following formula :
\begin{equation}\label{Radon-Fourier}
\mathcal F^\tau_{\sigma,\lambda} f (k)= \mathcal F^\tau_\sigma f (\lambda,k)=\mathcal F_\mathbb{R}\left[\mathcal R^\tau
_\sigma f(\cdot,k)\right](\lambda),\quad k\in K,
\end{equation}
 where $\mathcal{F}_\mathbb{R} \phi(\lambda)=\int_{\mathbb{R}}{\rm e}^{-i\lambda t}\phi(t){\rm d}t$  is the Euclidean Fourier transform on $\mathfrak a=\mathbb R$.

 %We come now to the statement of the    uniform continuity  estimate  for restriction of the Helgason Fourier transform.% operator

\begin{theorem}[Helgason Fourier restriction theorem]\label{fourier} Let $\tau\in\Lambda$ and $\sigma\in\widehat{M}(\tau)$.

$(1)$ Suppose $n$ even. There exists a positive constant $C$ such that for $\lambda\in\mathbb R\setminus\{0\}$ and $R>1$ we have 
\begin{equation}\label{fourier-rest} 
\|\mathcal F_{\sigma,\lambda}^{\tau} f \|_{L^2(K,\sigma)}\leq C R^{1/2} \nu(\lambda)^{-1/2}     \|f\|_{L^2(G,\tau)},
\end{equation}
for every $f\in L^2(G,\tau)$ supported in $B(R)$.

$(2)$ Suppose $n$ odd. There exists a positive constant $C$ such that for $\lambda\in\mathbb R\setminus\{0\}$ and $R>0$ we have
\begin{equation}\label{fourier-rest1} 
\|\mathcal F_{\sigma,\lambda}^{\tau} f \|_{L^2(K,\sigma)}\leq C R^{1/2} \nu(\lambda)^{-1/2}     \|f\|_{L^2(G,\tau)},
\end{equation}
for every $f\in L^2(G,\tau)$ supported in $B(R)$.
\end{theorem}
 
\begin{proof}
$(1)$ Suppose $n$ even. It is sufficient to prove \eqref{fourier-rest} for $f\in C_c^\infty(G,\tau)$. Using the Plancherel formula \eqref{Plan-Formula-even}, 
the proof is reduced to establishing that there exists   $C>0$ such that for any $R>1$  
\begin{align*}
\nu(\lambda)\parallel \mathcal F_{\sigma_{n-1},\lambda}^{\tau_n^\pm} f\parallel^2_{L^2(K,\sigma_{n-1})}\leq C R\int_0^\infty \nu(\lambda) \diff \lambda\, \parallel \mathcal F_{\sigma_{n-1},\lambda}^{\tau_n^\pm} f\parallel^2_{L^2(K,\sigma_{n-1})}, 
\end{align*}
for any  $f\in C_c^\infty(G,\tau)$ with $\operatorname{supp}\,f\subset B(R)$.\\
From \eqref{Radon-Fourier}  the above inequality is equivalent to 
\begin{equation}\label{R1}
 \parallel  \nu(\lambda)^{1/2}\mathcal F_\mathbb{R}[\mathcal R^{\tau_n^\pm}_{\sigma_{n-1}} f(t,\cdot)](\lambda)\parallel^2_{L^2(K,\sigma_{n-1})}\\
 \atop \hspace{5cm}\leq C R\int_0^\infty \nu(\lambda) \diff \lambda\, \parallel \mathcal F_\mathbb{R}[\mathcal R^{\tau_n^\pm}_{\sigma_{n-1}} f(t,\cdot)](\lambda)\parallel^2_{L^2(K,\sigma_{n-1})}. 
\end{equation} 
Using \eqref{dasy}  and according to  \cite[Lemma 4]{Anker} we may replace $\nu(\lambda)^{1/2}$ by 
$(1+ |\lambda|)^{\frac{n-1}{2}}$ and $(1+|\lambda|)^{\frac{n-1}{2}}$ by a Schwartz function $\eta_{\rho}(\lambda)$ on $\mathbb R$,  which is the Fourier transform of a tempered distribution $T$ with compact support.  Then  the  inequality  \eqref{R1} is equivalent to 
\begin{equation}\label{F2}
CR^{-1}\parallel \mathcal{F}_\mathbb{R}[T\ast \mathcal{R}^{\tau_n^\pm}_{\sigma_{n-1}} f(t,\cdot)](\lambda)\parallel^2_{L^2(K,\sigma_{n-1}^\pm)}
 \atop \hspace{6cm} \leq \int_\mathbb{R} \diff \lambda\, \parallel \mathcal{F}_\mathbb{R}[T\ast \mathcal{R}^{\tau_n^\pm}_{\sigma_{n-1}} f(t,\cdot)](\lambda)\parallel^2_{L^2(K,\sigma_{n-1}^\pm)}.
\end{equation}
To prove \eqref{F2},  suppose  $\operatorname{supp}  T\subset [-R_0,R_0]$ for some $R_0>1$. Then it  follows  from \eqref{RS} that
$$
\operatorname{supp}  T\ast \mathcal{R}^{\tau_n^\pm}_{\sigma_{n-1}} f(\cdot ,k)\subset [-(R_0+R),(R_0+R)],
$$ for all $k\in K$. Therefore by Cauchy-Schwarz inequality and the Euclidean Plancherel theorem, we have 
\begin{align*}\begin{split}
\parallel \mathcal{F}_\mathbb{R}[T\ast \mathcal{R}^{\tau_n^\pm}_{\sigma_{n-1}} f(t,\cdot)](\lambda)\parallel^2_{L^2(K,\sigma_{n-1}^\pm)}
\leq (R+R_0)\int_K   \diff k \int_\mathbb{R} \diff t\parallel T\ast \mathcal{R}^{\tau_n^\pm}_{\sigma_{n-1}} f(\cdot,k)(t)\parallel^2  \\
= (1+\frac{R_0}{R})R\int_K \diff k \int_\mathbb{R}\diff \lambda \parallel \mathcal{F}_\mathbb{R}[ T\ast \mathcal{R}^{\tau_n^\pm}_{\sigma_{n-1}} f(\cdot,k)](\lambda)\parallel^2 ,
\end{split}\end{align*}
%where we have used  the Euclidean Plancherel theorem, 
thus   the result follows.

$(2)$  Suppose $n$  odd. 
By the Plancherel formula \autoref{Plan-Formula-odd} we have
 \begin{equation*} 
\|f\|_{L^2\left(G, \tau_n\right)}^2=  \int_0^{\infty}   \nu(\lambda) \diff \lambda\,  
\left[\|\mathcal{F}^{\tau_n}_{\sigma_{n-1}^+,\lambda}f \|_{L^2\left(K, \sigma_{n-1}^{+}\right)}^2 
+\|\mathcal{F}^{\tau_n}_{\sigma_{n-1}^-,\lambda}f \|_{L^2\left(K, \sigma_{n-1}^{-}\right)}^2 
 \right] .
\end{equation*} 
It is therefore sufficient to show that there exists a positive constant $C$ such that for any $R>0$ and $f\in L^2(G,\tau_n)$ with  $\operatorname{supp}\,f\subset B(R)$, we have
\begin{equation}\label{sam27}
\parallel  \nu(\lambda)^{1/2}\mathcal{F}^{\tau_n}_{\sigma_{n-1}^\pm,\lambda} f\parallel^2_{L^2(K,\sigma_{n-1}^\pm)}\leq 
 CR\int_0^{\infty}   \nu(\lambda) \diff \lambda\,  
 \| \mathcal{F}^{\tau_n}_{\sigma_{n-1}^\pm,\lambda} f\|_{L^2(K,\sigma_{n-1}^\pm)}^2.
 \end{equation}
 To show \eqref{sam27}, let $f\in C^\infty_c(G,\tau)$ with $\operatorname{supp}\,f\subset B(R)$.  Since $\nu(\lambda)^{1/2}$ is a polynomial of degree $\frac{n-1}{2}$ it follows that
\begin{align*}\begin{split}
\parallel \nu(\lambda)^{1/2} \mathcal{F}^\tau_{\sigma,\lambda} f(k)\parallel &= \parallel \nu(\lambda)^{1/2} \mathcal{F}_\mathbb{R}\mathcal{R}^\tau_{\sigma} f(\cdot,k)(\lambda)\parallel\\
&=\parallel  \mathcal{F}_\mathbb{R}[\nu\left(\frac{d}{dt}\right)^{1/2}( \mathcal{R}^\tau_{\sigma} f(\cdot,k)](\lambda)\parallel\\
&\leq \int_\mathbb{R}  \diff t\parallel [\nu\left(\frac{d}{dt}\right)^{1/2}( \mathcal{R}^\tau_\sigma f(t,k)]\parallel .
\end{split}\end{align*}
 As $\nu\left(\frac{d}{dt}\right)^{1/2}$ is a differential operator we have  $\operatorname{supp}\,[\nu\left(\frac{d}{dt}\right)^{1/2}  \mathcal{R}^\tau_\sigma f(\cdot,k)]\subset   [-R,R]$. Then the Cauchy-Schwarz inequality and the Euclidean Plancherel formula imply
 $$\begin{aligned}
	\parallel \nu\left(\lambda\right)^{1/2} \mathcal{F}^\tau_\lambda f(k)\parallel^2 
	&\leq 2 R\int_\mathbb{R} \diff \lambda\,
	\Bigl\| 
	\mathcal{F}_\mathbb{R}\Bigl[\nu\left(\frac{d}{dt}\right)^{1/2}\mathcal{R}^\tau_\sigma  f(\cdot,k)\Bigr]
	\Bigr\|^2\\
&=2R\int_{\mathbb R}\nu(\lambda) \diff \lambda\, \|\mathcal F_{\sigma,\lambda}^\tau f(k)\|^2.
\end{aligned}$$  
Thus, integrating both part of the above inequality with respect to $k$ the result follows. This completes the proof of Theorem \ref{fourier}. 
\end{proof}

\section{Poisson transform on spinors}
  Let  $\sigma\in \widehat M$ be a spinor  or half-spinor representation of $M=\Spin(n-1)$. Recall that the space of spinors on $K/M$ can be realized as the homogeneous vector bundle $K\times_M V_{\sigma}$.  
The space of sections if this bundle is  identified with $\Gamma(K,\sigma)$,  the space of $V_\sigma$-valued functions  $F$ on $K$ which are $M$-covariant with respect to $\sigma$, meaning  $F(km)=\sigma(m)^{-1}F(k)$ for all $m\in M$ and $k\in K$. In what follows, we denote by  $C^\infty(K,\sigma)$, $C^{-\omega}(K,\sigma )$ and $L^2(K,\sigma)$  the spaces of smooth sections, hyperfunctions sections and  $L^2$-sections (with respect to the normalized Haar measure of $K$), respectively.

For any $\lambda\in \mathfrak a^*_{\mathbb C}\simeq \mathbb C$, let    $\sigma_{\lambda}$  denote the representation of the minimal parabolic subgroup $ P=MAN$ of $G$   given by
$$\sigma_{\lambda}(ma_tn)=e^{(\rho-i\lambda) t}\sigma(m).$$
 The group $G$ acts on the space
$$L^2(G;\sigma_\lambda):=
\{f\colon G\to V_\sigma , f(gma_tn)=e^{(i\lambda-\rho)t}\sigma(m)^{-1} f(g) ,\; f_{\mid K}\in L^2(K) \},$$
by left translations, $\pi_{\sigma,\lambda} (g) f(h) =f(g^{-1} h)$. These are   the   principal series representations of $G$.
It is known that as a $K$-module, $L^2(G,\sigma_\lambda)$ is isomorphic (for any $\lambda$ ) to the space   $L^2(K,\sigma)$  
on which $\pi_{\sigma,\lambda}$ acts by
\begin{equation}\label{princ-series}
\pi_{\sigma, \lambda}(g) f(k)=e^{(i \lambda-\rho) H\left(g^{-1} k\right)} f\left(\kappa\left(g^{-1} k\right)\right). 
\end{equation}
  
%Recall that the space $L^2(K,\sigma)$   is   identified with $L^2$-sections of the bundle $K\times_M V_\sigma$ over $\partial H^n(\mathbb R)=K/M$.

 Let $\tau\in\Lambda$ and  $\sigma\in\widehat M(\tau)$.  
For  $\lambda\in\mathfrak a^*_{\mathbb C}\simeq \mathbb C$, the Poisson transform  on $L^2(K,\sigma)$ is the map 
\begin{align*}
\mathcal{P}_{\sigma,\lambda}^\tau \colon L^2(K,\sigma)\longrightarrow  C^\infty(G,\tau)
\end{align*}
given by\footnote{The parameter $\lambda$ in \cite{CP} corresponds to  $-\lambda$ in our paper.}
\begin{equation}\label{Poisson}
\mathcal{P}_{\sigma,\lambda}^\tau\, F(g)= \sqrt{d_{\tau,\sigma} }
 \int_K \diff k\,{\rm e}^{-(i\lambda+\rho)H(g^{-1}k)}\tau(\kappa(g^{-1}k))   F(k).
\end{equation}

 Since   for $\lambda\in \mathbb{R}$, the Poisson transform  $\mathcal P_{\sigma,\lambda}^\tau$ can be seen as the adjoint \hb{of} $\mathcal F_{\sigma,\lambda}^\tau$,      the following uniform estimate for Poisson transform is an immediate consequence of Theorem \eqref{fourier}.
 
\begin{proposition}\label{pro-est-Poisson}
 Let $\tau\in\Lambda$ and $\sigma\in\widehat M(\tau)$.  

(1) Suppose $n$ even. Then there exists a positive constant $C$ such that for $\lambda\in\mathbb R\setminus\{0\}$ we have
\begin{equation}\label{esti-Poisson-enen}
\sup_{R>1} \frac{1}{R}\int_{B(R)} \diff (gK)\,\|\mathcal P_{\sigma,\lambda}^\tau F(g)\|_{\tau}^2   \leq C  \nu(\lambda)^{-1} \|F\|^2_{L^2(K,\sigma)},
\end{equation}
for any $F\in L^2(K,\sigma)$.

(2) Suppose $n$ odd. Then there exists a positive constant $C$ such that for $\lambda\in\mathbb R\setminus\{0\}$ we have
\begin{equation}\label{esti-Poisson-odd}
\sup_{R>0}  \frac{1}{R}\int_{B(R)}\diff (gK)\, \|\mathcal P_{\sigma,\lambda}^\tau F(g)\|_{\tau}^2   \leq C  \nu(\lambda)^{-1} \|F\|^2_{L^2(K,\sigma)},
\end{equation}
for any $F\in L^2(K,\sigma)$.
\end{proposition}

 Let $\mathbb D(G,\tau)$ be   the left algebra  of $G$-invariant differential operators acting on $C^\infty(G,\tau)$. Following  \cite{Gaillard}, $\mathbb D(G,\tau)$ is a commutative algebra, furthermore  $\mathbb D(G,\tau^{\pm}_n)\simeq \mathbb{C}[\fsl D^2]$ if $n$ is even and   $\mathbb D(G,\tau_n)\simeq \mathbb{C}[\fsl D]$  if $n$ is odd.  Here $\fsl D$ is the Dirac operator, which can be expressed in local coordinates as
  $$\fsl D=\sum_{i=1}^n \tau(e_i) \partial_{e_i},$$
 where $(e_i)_{i=1,\ldots n}$ is a local orthonormal  basis of $T_oH^n(\mathbb R)\simeq \mathbb R^n$.  
 
It is known (see \cite[Theorem 5.1]{CP}) that for $\lambda\in \mathbb C,$ we have:
\begin{equation*}%\label{eigen1}
\fsl D^2\,\mathcal{P}_{\sigma_{n-1},\lambda}^{\tau_n^\pm} F=\lambda^2 \mathcal{P}_{\sigma_{n-1},\lambda}^{\tau_n^\pm} F, \; \text{ for any }\, F\in  C^\infty(K,\sigma_{n-1})\quad \text{if $n$ is even},
\end{equation*}
 \begin{equation*}%\label{eign2}
\fsl D\, \mathcal{P}_{\sigma_{n-1}^\pm,\lambda}^{\tau_n} F=\mp \lambda\, \mathcal{P}_{\sigma_{n-1}^\pm,\lambda}^{\tau_n} F,\; \text{ for any } F\in  C^\infty(K,\sigma^\pm_{n-1})\quad \text{if $n$ is odd}.
\end{equation*}
 Hence, 
 $$\begin{aligned}
\mathcal P_{\sigma_{n-1},\lambda}^{\tau_n^\pm} &\colon L^2(K,\sigma_{n-1}) \to \mathcal{E}_{\sigma_{n-1},\lambda}(G, \tau_n^\pm) \quad \text{if $n$ is even},\\
\mathcal P_{\sigma_{n-1}^\pm,\lambda}^{\tau_n} &\colon L^2(K,\sigma_{n-1}^\pm) \to \mathcal{E}_{\sigma_{n-1}^\pm,\lambda}(G, \tau_n)  \quad \text{if $n$ is odd},
\end{aligned}
$$
where  the target  spaces are the following   eigenspaces 
 \begin{align} 
\mathcal{E}_{\sigma_{n-1},\lambda}(G, \tau_n^\pm) &=\left\{ F\in C^\infty(G,\tau_n^\pm) \mid \fsl D^2 F =  \lambda^2 F\right\}& \,\text{if $n$ is even},\label{g1}\\
 \mathcal{E}_{\sigma_{n-1}^\pm, \lambda}(G, \tau_n) &=\left\{ F\in C^\infty(G,\tau_n) \mid \fsl D F =  \mp\lambda F\right\}& \,\text{if $n$ is odd}.\label{g2}
\end{align} 
  %To get a more accurate on the image of $\mathcal P_{\sigma,\lambda}^\tau$,
  Now,  Proposition \ref{pro-est-Poisson}  is  intended to motivate introducing  the space 
  $\mathcal E_{\sigma,\lambda}^2(G,\tau)$  given by %one of the following spaces according  whether $n$ is even or odd :
$$\begin{aligned}
\mathcal E_{\sigma_{n-1},\lambda}^2(G,\tau_n^\pm )&=\mathcal E_{\sigma_{n-1},\lambda}(G,\tau_n^\pm)\cap B^*(G,\tau_n^\pm) \quad \text{if $n$ is even},\\
 \mathcal E_{\sigma_{n-1}^\pm,\lambda}^2(G,\tau_n )&=\mathcal E_{\sigma_{n-1}^\pm,\lambda}(G,\tau_n )\cap B^*(G,\tau_n ) \quad \text{if $n$ is odd},
 \end{aligned}$$
  where
   $B^*(G,\tau)$  is the space of   $f\in L^2_{\text{loc}}(G,\tau)$ such that
 \begin{equation*}%\label{norm*}
 \|f \|^2_*=\begin{cases}
 \displaystyle\sup_{R>1} \frac{1}{R}\int_{B(R)}\diff (gK)\,\|f(g)\|^2_\tau  <\infty &  \text{if $n$ is even},\\
   \displaystyle\sup_{R>0} \frac{1}{R}\int_{B(R)}\diff (gK)\,\|f(g)\|^2_\tau  <\infty &  \text{if $n$ is odd}.
 \end{cases}
\end{equation*}
 It follows immediately  form  Proposition \ref{pro-est-Poisson} and the irreducibility of the principal series representations $\pi_{\sigma,\lambda}$ for $\lambda\in \mathbb{R}\setminus \{0\}$  that,  the Poisson transforms  
\begin{equation}\label{p1}
	 \mathcal P_{\lambda,\sigma_{n-1}}^{\tau_n^\pm} \colon L^2(K,\sigma_{n-1}) \to \mathcal E_\lambda^2(G,\tau_n^\pm ) \quad \text{if $n$ is even},
\end{equation}
\begin{equation}\label{p2}
 \mathcal P_{\lambda,\sigma_{n-1}^\pm}^{\tau_n}  \colon L^2(K,\sigma_{n-1}^\pm) \to \mathcal E_{\pm,\lambda}^2(G,\tau_n ) \quad \text{if $n$ is odd}.
\end{equation}
are continuous injective linear maps for every $\lambda\in \mathbb{R}\setminus \{0\}$. We will 
   show below, that   \eqref{p1} and \eqref{p2} are indeed  bijective maps. 
 
\section{Announcement of the main results}
 
Let $\tau\in\Lambda$ and $\sigma\in\widehat M(\tau)$.  For $f \in L^2(G,\tau)$ we define the function 
 $\mathcal{Q}^\tau_{\sigma,\lambda} f \in C^{\infty}(G,\tau)$ for a.e. $\lambda \in \mathbb R$ by
  \begin{equation}\label{proj-Q1}
 \mathcal{Q}^\tau_{\sigma,\lambda} f(g)=\sqrt{ d_{\tau,\sigma}  }\, {\nu(\lambda)} \int_K \diff k \,
  e^{-(i\lambda+\rho)H(g^{-1}k)} \tau(\kappa(g^{-1}k))
 \mathcal{F}_{\sigma,\lambda}^\tau f(k).
 \end{equation}
Then, according to the Helgason Fourier inversion formula (\eqref{Inv-Formula-even}-\eqref{Inv-Formula-odd})  we have the following spectral decomposition  
 $$f=   \sum_{\sigma\in\widehat{M}(\tau}\int_0^\infty    \diff \lambda\, \mathcal Q_{\sigma, \lambda}^\tau f,  $$
The family of operators
$$\mathcal{Q}^\tau_{\lambda}  =\left(\mathcal{Q}^\tau_{\sigma,\lambda} \right)_{\sigma\in\widehat{M}(\tau)} : L^2(G,\tau)\to \oplus_{\sigma\in\widehat{M}(\tau)}\mathcal E_{\sigma,\lambda}(G,\tau),\quad \lambda\in\mathbb R_+$$
  %The family    $\left(\mathcal{Q}^\tau_{\lambda}\right)_{\lambda \in \mathbb R_{+}}$ 
  are called    generalized spectral projections.   
  We also set
$$\mathcal Q^\tau f(\lambda,g)=\mathcal Q^\tau_\lambda f(g) \; \text{for a.e. } (\lambda,g)\in\mathbb R_+\times G.$$

Our objective is therefore to investigate an image characterization of $Q^\tau$. To accomplish this, we begin by defining the space
%To this end, we begin by introducing   the space
$$\mathcal E^2_{\mathbb R_+}(G,\tau):=\oplus_{\sigma\in\widehat M(\tau)} \mathcal E^2_{\sigma,\mathbb R_+}(G,\tau)$$ where $\mathcal E^2_{\sigma,\mathbb R_+}(G,\tau)$ is   the space of $V_\tau$-valued measurable functions   $\psi$ on $\mathbb R_+\times G$ such that $\psi(\lambda,\cdot)\in \mathcal E_{\sigma,\lambda}(G,\tau)$   for a.e. $\lambda\in(0,\infty)$ and $\| \psi\|_+<\infty$, with
\begin{equation}\label{norm+}
    \| \psi\|_{+}^2 :=
\begin{cases}
\displaystyle\sup_{R>1} \int_0^\infty \diff\lambda\,\frac{1}{R}\int_{B(R)}\diff (gK)\, \|\psi(\lambda,g)\|_\tau^2  <\infty & n \; \text{even},\\
\displaystyle\sup_{R>0} \int_0^\infty \diff\lambda\,\frac{1}{R}\int_{B(R)} \diff (gK)\,\|\psi(\lambda,g)\|_\tau^2 
   <\infty & n \; \text{odd}.
\end{cases}
\end{equation}
We also set for   $\phi= \left( \phi_\sigma\right)_{\sigma\in\widehat{M}(\tau)}$ 
$$\|\phi \|_+^2 = \sum_{\sigma\in\widehat M(\tau)} \|\phi_\sigma\|_+^2.$$

 Now we are ready to  state our   main theorems.

\begin{theorem}[Strichartz's conjecture for the spectral projections]\label{main-th-proj}% 
Let $\tau\in \Lambda$ and $\sigma\in \widehat{M}(\tau)$.

$(1)$  
 There exists a positive constant $C$   such that  for any $f\in L^2(G,\tau)$, we have
\begin{equation}\label{esti-main1}
C^{-1} \|f\|_{L^2(G,\tau )} \leqslant\left\|\mathcal{Q}^\tau f\right\|_{+} \leqslant C \|f\|_{L^2(G,\tau)} .
\end{equation}
  Furthermore, we have
\begin{equation}\label{asym-Q}
\lim _{R \rightarrow \infty} \int_0^\infty \diff\lambda\,\frac{1}{R} \int_{B(R)}\diff (gK)\, \left\|\mathcal{Q}_{ \lambda}^\tau f(g)\right\|_{\tau}^2   = \gamma_0 \|f\|_{L^2(G,\tau)}^2 .
\end{equation}
where $\gamma_0=2/\pi$ for $n$ odd and $\gamma_0=4/\pi$ for $n$ even.

$(2)$  The linear map $\mathcal{Q}^\tau$ is a topological isomorphism from $L^2(G,\tau)$ onto $\mathcal{E}_{\mathbb R_+}^2(G,\tau)$.

   \end{theorem}

 Notice that  the spectral projection  $\mathcal{Q}_{\sigma, \lambda}^\tau f$  can be written as 
$$
\mathcal{Q}_{\sigma, \lambda}^\tau f(g)=   {\nu(\lambda)}   \mathcal{P}_{\sigma, \lambda}^\tau\left(\mathcal{F}_{\sigma, \lambda}^\tau f\right)(g)
\quad \text{for a.e. } \lambda \in \mathbb{R}.$$
Thus, above theorem will be a   consequence of the following:

\begin{theorem}[Strichartz's conjecture for the Poisson transform] \label{main-th-Poisson} 
Let $\tau\in \Lambda$, $\sigma\in \widehat{M}(\tau)$ and $\lambda\in\mathbb R\setminus\{0\}$.

$(1)$  
 There exists a positive constant $C$ independent of $\lambda$ such that  for any $f\in L^2(K,\sigma)$,
\begin{equation}\label{esti-Poisson16}
C^{-1} \nu(\lambda)^{-1/2}\|f\|_{L^2(K,\sigma )} \leqslant\left\|\mathcal{P}_{\sigma,\lambda}^\tau f\right\|_* \leqslant C \nu_\sigma(\lambda)^{-1/2}\|f\|_{L^2(K,\sigma)} .
\end{equation}
  Furthermore, we have
\begin{equation}\label{asym-Poisson16}
\lim _{R \rightarrow \infty} \frac{1}{R} \int_{B(R)} \diff (gK)\,\left\|\mathcal{P}_{\sigma,\lambda}^\tau f(g)\right\|_{\tau}^2 
=\gamma_0  \nu_\sigma(\lambda)^{-1}\|f\|_{L^2(K,\sigma)}^2 .
\end{equation}
$(2)$  The Poisson transform $\mathcal{P}_{\sigma,\lambda}^\tau$ is a topological isomorphism from $L^2(K,\sigma)$ onto $\mathcal{E}_{\sigma,\lambda}^2(G,\tau)$.

\end{theorem}

%The rest of the paper  will be devoted to   the proof of    Theorem \ref{main-th-Poisson} and then Theorem \ref{main-th-proj}.

The rest of the paper is devoted to providing a   proof of Theorem \ref{main-th-Poisson}, followed by a proof of Theorem \ref{main-th-proj}.

\section{Asymptotic expansion for $\tau$-spherical functions} 

Let $(\delta, V_\delta)$ be a   an irreducible representation of $K$.
A function $\Phi\colon G\to \operatorname{End}(V_\delta)$ with   $\Phi(e)=\operatorname{Id}$ is  called $\delta$-spherical if

\quad $(a)$ $\Phi$ is $\delta$-radial, i.e.,
$$ \Phi(k_1gk_2)=\delta(k_2)^{-1} \Phi(g) \delta(k_1)^{-1},\,\; \forall k_1, k_2\in K,\forall g\in G;$$

\quad $(b)$  $\Phi(\cdot)v$ is an eigenfunction for the algebra $\mathbb D(G,\delta)$ for one nonzero
$v\in V_\delta$ (hence for all $v\in V_\delta$).

 Let $\tau\in\Lambda$ and $\sigma\in\widehat M(\tau)$. 
 For $\lambda\in \mathbb{C}$,  define the Eisenstein integral
\begin{equation}\label{spherical}
 { \Phi_{\sigma,\lambda}^\tau(g)}  
 = d_{\tau,\sigma}
 \int_K \diff k\,{\rm e}^{-(i\lambda+\rho)H(g^{-1}k)}\tau(\kappa(g^{-1}k))   P_\sigma\tau(k)^{-1}.
\end{equation}
%\kk{ where $P\sigma$ is the projection of $V_\tau$ onto $V_\sigma$}. 
Then $\Phi_{\sigma,\lambda}^\tau$ is a $\tau$-spherical function on $G$, corresponding to the principal series representation $\pi_{\sigma,\lambda}$, see \eqref{princ-series}. 
%=\operatorname{Ind}_{MAN}^G(\sigma\otimes e^{\rho-i\lambda}\otimes 1)$.
Moreover any $\tau$-spherical function is given by \eqref{spherical}, for some $(\sigma,\lambda)\in    \widehat M(\tau)\times \mathbb{C}$.\\
For   $\Re(i\lambda)>0$, we have the following asymptotic behavior (see e.g. \cite[Proposition 5.2]{BBK2})  
\begin{equation}\label{asym-comp}
\lim_{t\rightarrow \infty}{\rm e}^{(\rho-i\lambda)t} \Phi_{\sigma,\lambda}^\tau(a_t)=d_{\tau,\sigma}\mathbf{c}(\lambda,\tau)P_\sigma, 
\end{equation}
where $\mathbf c(\lambda,\tau)$ is the generalized Harish-Chandra $c$-function given by  
\begin{equation}\label{HC-function}
\mathbf c(\lambda,\tau)
=\int_{\overline{N}}\diff\overline{n}\,{\rm e}^{-(i\lambda+\rho)H(\overline{n})}\tau(\kappa(\overline{n})) \in \mathrm{End}_M(V_\tau).
\end{equation}
Above ${\rm d}\overline n$ denotes the Haar measure on $\overline  N=\theta(N)$ with the normalization $\int_{\overline N} {\rm e}^{-2\rho(H(\overline n))} {\rm d}\overline n=1$.
The integral \autoref{HC-function} converges for $\lambda\in \mathbb C$ such that \(\Re(i\lambda)>0\) and it has a meromorphic continuation to \({\C}\).\

By \cite[Proposition 5.4]{BBK2} we have
$$\begin{aligned}
	\mathbf c(\lambda,\tau_n^\pm)&=\mathbf{c}(\lambda) \mathrm{Id}_{V_{\sigma_{n-1}}}&\text{for}\; n \text{ even},\\
	\mathbf c(\lambda,\tau_n)&=\frac{1}{2}\mathbf{c}(\lambda) \mathrm{Id}_{V_{\sigma_{n-1}^+}} +\;
	\frac{1}{2}\mathbf{c}(\lambda) \mathrm{Id}_{V_{\sigma_{n-1}^-}} &\text{for}\; n \text{ odd},
\end{aligned}$$
where  the scalar   component $\mathbf{c}(\lambda)$ is given in \eqref{c-function1}. 
%Hence, for any $\sigma\in\widehat M(\tau)$, we have
 Therefore, it follows from \eqref{asym-comp} that  
\begin{equation}\label{c-proj}
 \frac{\dim \tau}{\dim\sigma} \mathbf c(\lambda,\tau)P_\sigma = \mathbf{c}(\lambda)P_\sigma,
\end{equation}
 for any $\sigma\in\widehat M(\tau)$. 
%Notice that the  asymptotic \eqref{asym-comp} is no longer  true for $\lambda$ real   because of the oscillating terms at infinity in $  \Phi_{\sigma,\lambda}^\tau(a_t)$.

 According to the Cartan decomposition, $\Phi_{\sigma,\lambda}^\tau$ is completely determined by its restriction to $A$. Since $\Phi_{\sigma,\lambda}^\tau(a_t)\in \mathrm{End}_M(V_\tau)$, it follows by Schur's lemma that $\displaystyle\Phi_{\sigma,\lambda}^\tau(a_t)$ is scalar on each $M$-isotypical  component    of $V_\tau$. Thus, 
 %\kk{\sout{there exists functions $f_{\eta,\lambda}:\mathbb{R}\to \mathbb{C}$ such that $\Phi_{\sigma,\lambda}^\tau(a_t)=\sum_{\eta\in \widehat{M}(\tau)}f_{\eta,\lambda}(t)\operatorname{Id}_{V_\eta}$. Then we set}}
\begin{align*}
	\Phi_{\sigma_{n-1},\lambda}^{\tau_n^\pm}(a_t)&=\varphi^\pm(\lambda,t) \mathrm{Id}_{V_{\sigma_{n-1}}}&& \text{whenever}\; n \;\text{is even},\\
	\Phi_{\sigma_{n-1}^\pm,\lambda}^{\tau_n}(a_t)&=\varphi_{\pm}^+(\lambda,t) \mathrm{Id}_{V_{\sigma_{n-1}^+}} +\;
	\varphi_{\pm}^-(\lambda,t) \mathrm{Id}_{V_{\sigma_{n-1}^-}}&&\text{whenever}\; n \;\text{is odd}.
\end{align*}
The scalar components  $\varphi^\pm$, $\varphi_\pm^+$ and $\varphi_\pm^-$  are given in terms of the Jacobi function 
 \begin{equation*}%\label{jacobi}
 \phi_\lambda^{(\alpha,\beta)}(t)={}_2F_1\left(\frac{i\lambda+\alpha+\beta+1}{2},\frac{-i\lambda+\alpha+\beta+1}{2};\alpha+1; -\sinh^2 t\right),  
\end{equation*}
where  $\alpha, \beta, \lambda\in \mathbb{C}$ with $\alpha\neq -1,-2, \ldots$,  (see, e.g. \cite{CP}\cite{Ko}). 
More precisely, following \cite[Theorem 5.4]{CP}, we have :
\begin{enumerate}[\upshape (1)]
\item   When $n$ is even,   
\begin{equation*}%\label{scal-phi-even}
 \varphi^{\pm}(\lambda, t)=\left(\cosh\frac{t}{2}\right) \phi_{2 \lambda}^{(n / 2-1, n / 2)}\left(\frac{t}{2}\right).  
\end{equation*}
\item   When $n$ is odd,  
\begin{equation*}%\label{scal-phi-odd1}
 \varphi_{\pm}^{+}(\lambda, t) =\left(\cosh \frac{t}{2}\right) \phi_{2 \lambda}^{(n / 2-1, n / 2)}\left(\frac{t}{2}\right) \pm i \frac{2 \lambda}{n}\left(\sinh \frac{t}{2}\right) \phi_{2 \lambda}^{(n / 2, n / 2-1)}\left(\frac{t}{2}\right),	
 \end{equation*}
\begin{equation*}%\label{scal-phi-odd2}
 	\varphi_{\pm}^{-}(\lambda, t) =\left(\cosh \frac{t}{2}\right) \phi_{2 \lambda}^{(n / 2-1, n / 2)}\left(\frac{t}{2}\right) \mp i \frac{2 \lambda}{n}\left(\sinh \frac{t}{2}\right) \phi_{2 \lambda}^{(n / 2, n / 2-1)}\left(\frac{t}{2}\right).
 \end{equation*}
  \end{enumerate}
 
 Recall  that   the function $\phi_\lambda^{\alpha,\beta}$ has the following development (see e.g. \cite{Ko, FJ})
  \begin{equation}\label{key-F-2}
\phi^{(\alpha, \beta)}_\lambda(t)=c_{\alpha, \beta}( \lambda) \Psi_{\lambda}^{\alpha, \beta}(t)+ c_{\alpha, \beta}(- \lambda) \Psi_{- \lambda}^{\alpha, \beta}(t), 
\end{equation}
where %$c_{\alpha, \beta}(\lambda)$ is given by 
$$
c_{\alpha, \beta}(\lambda)=\frac{2^{-i \lambda+\alpha+\beta+1} \Gamma(\alpha+1) \Gamma(i \lambda)}{\Gamma\left(\frac{i \lambda+\alpha+\beta+1}{2}\right) \Gamma\left(\frac{i \lambda+\alpha-\beta+1}{2}\right)},
$$
and  %$\Psi_\lambda^{\alpha,\beta}(t)$  is given
$$
\Psi_\lambda^{\alpha, \beta}(t)=(2 \sinh t)^{i \lambda-\alpha-\beta-1}{ }_2 F_1\left(\frac{\alpha+\beta+1-i \lambda}{2}, \frac{-\alpha+\beta+1-i \lambda}{2} ; 1-i \lambda ;-\frac{1}{\sinh ^2 t}\right) .
$$
Furthermore,
there exists a constant $C>0$ such that
\begin{equation}\label{theta-1}
\Psi_\lambda^{\alpha, \beta}(t)=e^{\left(i \lambda-\alpha-\beta-1 \right) t}\left(1+\mathrm{e}^{-2 t} \Theta(\lambda,t)\right), \text { with } \quad\left|\Theta(\lambda,t)\right| \leq C .
\end{equation}
  for all $\lambda \in \mathbb{R}$ and all $t \geq 1$.

 The main result of this section is to establish, for real parameter $\lambda$,   an asymptotic expansion of the translated $\tau$-spherical functions $\Phi_{\sigma,\lambda}^\tau(g^{-1}x)$, ($x, g\in G$).  
  Before stating  the result let us say,     that $f_1, f_2 \in B^*(G,\tau)$ are asymptotically equivalent, and we denote $f_1\simeq f_2$ if
\begin{equation}\label{asym-means} 
\lim_{R\to \infty} \frac{1}{R}\int_{B(R)} \diff(gK)\,\|f_1(g)-f_2(g)\|_\tau^2   =0.
 \end{equation}

 Let $ {M}^{\prime}$ be the normalizer of $A$ in ${K}$. Then the Weyl group $W={M}^{\prime}/{M}$ has two elements. We denote by $\omega$ a representative of the non-trivial element of $W$  and observe that $\omega\cdot H_0=-H_0$.  Then $\omega$ 
 stabilizes $\sigma_{n-1}$ (for $n$ even) and interchanges $\sigma_{n-1}^+$ and $\sigma_{n-1}^-$ (for $n$ odd).  
   Furthermore,
   $$\begin{aligned}
	\omega\cdot\pi_{\sigma_{n-1},\lambda}&=\pi_{\sigma_{n-1},-\lambda}    \quad  \text{ if $n$ is even},\\
\omega\cdot\pi_{\sigma^{\pm}_{n-1},\lambda}&=\pi_{\sigma^{\mp}_{n-1},-\lambda} \quad \text{ if $n$ is odd}.
\end{aligned}$$
Hence,  we have
\begin{equation}\label{equal}
\begin{aligned}
\Phi^{\tau_n^{\pm}}_{\sigma_{n-1},-\lambda}&=\Phi^{\tau_n^{\pm}}_{\sigma_{n-1},\lambda} & \text{ if $n$ is even},\\
\Phi_{\sigma_{n-1}^{\pm},-\lambda}^{\tau_n}&=\Phi_{\sigma_{n-1}^{\mp},\lambda}^{\tau_n} & \text{ if $n$ is odd}.	
\end{aligned}
\end{equation}

%Also, we  recall  that the non-trivial Weyl group element $w$ stabilizes $\sigma_{n-1}$ (for $n$ even) and interchanges $\sigma_{n-1}^+$ and $\sigma_{n-1}^-$ (for $n$ odd).  
  \begin{proposition}
    Let $\lambda\in\R\setminus\{0\}$. Then for any $x\in G$ we have the following asymptotic expansion   in $B^*(G,\tau)$,
 \begin{equation}\label{main-formula1}
 \Phi_{\sigma,\lambda}^\tau (x)v \simeq   d_{\tau,\sigma}\tau^{-1}(k_2(x)) \sum_{s\in W} e^{(is\lambda-\rho)A^+(x)} \mathbf{c}(\tau,s\lambda)   P_{s\sigma}\tau^{-1}(k_1(x))  v, 
 \end{equation}
 where $v\in V_\tau$ and $x=k_1(x) e^{A^+(x)}k_2(x)$.
  \end{proposition}
  
  \begin{proof}  
 Let us first  show that the right hand-side of \eqref{main-formula1} belongs to $ B^\ast(G,\tau)$.  
 Using \eqref{c-proj} we easily get 
\begin{align*}
\begin{split}
\frac{1}{R}
\int_{B(R)}\diff x\parallel [\tau^{-1}(k_2(x)) &\sum_{s\in W} e^{(is\lambda-\rho)A^+(x)} d_{\tau,\sigma}\,\mathbf{c}(\tau,s\lambda)   P_{s\sigma}\tau^{-1}(k_1(x))]v\parallel^2\\
&\leq 
\frac{2\parallel v\parallel^2}{R}
\mid \mathbf{c}(\lambda)\mid^2
 \int_0^R \diff t\,{\rm e}^{-2\rho t} (2\sinh t)^{2\rho}\\
&\leq 2 \parallel v\parallel^2\mid \mathbf{c}(\lambda)\mid^2,
\end{split}\end{align*}
for any $R>0$. 
 
 Now we prove the asymptotic expansion. We have 
\begin{align}
\begin{split}
\int_{B(R)}\diff x &\parallel \Phi_{\sigma,\lambda}^\tau(x)v-d_{\tau,\sigma}\tau^{-1}(k_2(x)) \sum_{s\in W} e^{(is\lambda-\rho)A^+(x)} \mathbf{c}(\tau,s\lambda)   P_{s\sigma}\tau^{-1}(k_1(x))v\parallel^2\\
&=\int_0^R   \int_K  \parallel[ \Phi_{\sigma,\lambda}^\tau(a_t)- \sum_{s\in W} e^{(is\lambda-\rho)t} \mathbf{c}(s\lambda) P_{s\sigma}]\tau^{-1}(k)v\parallel^2   (2\sinh t)^{2\rho}\diff k\diff t\,
\end{split}\end{align}
Hence,  \eqref{main-formula1} is an immediate consequence of following estimate: for $\lambda\in \mathbb{R}\setminus\{0\}$ and $t\geq 0$, we have 
\begin{equation}\label{asympt}
\parallel\Phi^\tau_{\sigma,\lambda}(a_t)-d_{\tau,\sigma}\sum_{s\in W} e^{(is\lambda-\rho)t}  \mathbf{c}(\tau,s\lambda) P_{s\sigma}\parallel\leq  c_0 \mid \mathbf{c}(\lambda)\mid{\rm e}^{-\rho-t} 
\end{equation}
%\end{proposition}
 for some $c_0>0$  independent o $\lambda$.\\
%\begin{proof}
By   continuity it is sufficient to prove the above estimate \eqref{asympt} for $t\geq 1$.\\
Suppose  $n$ odd. We first deal with $\sigma=\sigma_{n-1}^+$. 
Using \eqref{key-F-2} \hb{, \eqref{theta-1}} and noting that 
$$\displaystyle \mathbf{c}(\lambda)=\mathbf{c}_{\frac{n}{2}-1, \frac{n}{2}}(2\lambda)=\frac{n}{i2\lambda}\mathbf{c}_{\frac{n}{2}-1, \frac{n}{2}}(2\lambda),$$
  we may rewrite $\varphi_{\sigma_{n-1}^+}^\pm$ as 
\begin{equation*}\begin{split}
\varphi_{\sigma_{n-1}^+}^\pm(\lambda,t)&=\cosh(t/2)\sum_{s\in W}\mathbf{c}(s\lambda){\rm e}^{(is2\lambda-n)t/2}(1+{\rm e}^{-t}\Theta_1(s\lambda,t))\\
&\pm \sinh(t/2)\sum_{s\in W}\frac{\mathbf{c}(s\lambda)}{s}{\rm e}^{(is2\lambda-n)t/2}(1+{\rm e}^{-t}\Theta_2(s\lambda,t))),
\end{split}\end{equation*}
with $\mid \Theta_i(\lambda,t)\mid\leq C_i, i=1, 2$ for all $t\geq 1$. 
This implies that there exists a constant $c_0>0$ such that for all $\lambda\in \mathbb{R}$ and all $t\geq 1$ we have  
$$
\varphi_{\sigma_{n-1}^+}^\pm(\lambda,t)=\sum_{s\in W}\frac{1}{2}\mathbf{c}(s\lambda)(1\pm \frac{1}{s}){\rm e}^{(is\lambda-\rho)t}+{\rm e}^{-t-\rho}\Theta^\pm(\lambda,t),
$$
with $\mid \Theta^\pm(\lambda,t)\mid\leq c_0\mid \mathbf{c}(\lambda)\mid$.

Now, since $\displaystyle\Phi^{\tau_n}_{\sigma_{n-1}^+}(a_t)= \varphi_{\sigma_{n-1}^+}^+(\lambda,t)P_{\sigma_{n-1}^+}+\varphi_{\sigma_{n-1}^+}^-(\lambda,t)P_{\sigma_{n-1}^-}$ we get
$$\begin{aligned}
\Phi^{\tau_n}_{\sigma_{n-1}^+}(a_t)&=\mathbf{c}(\lambda){\rm e}^{(i\lambda-\rho)t}P_{\sigma_{n-1}^+}+\mathbf{c}(-\lambda){\rm e}^{(-i\lambda-\rho)t}P_{\sigma_{n-1}^-}+{\rm e}^{-t-\rho}[\Theta^+(\lambda,t)P_{\sigma_{n-1}^+}+\Theta^-(\lambda,t)P_{\sigma_{n-1}^-}],
\end{aligned}
$$
and the result follows.

The proof in the case $\sigma=\sigma_{n-1}^-$ can be handled in much the same way as for $\sigma=\sigma_{n-1}^+$. When $n$ is even, the proof is similar to the above, so we omit it.

\end{proof}

%\kk{\sout{Next, we will prove the following asymptotic for the translated $\tau$-spherical functions.}}

\begin{proposition}[Key formula]\label{key-formula} Let $\tau\in \Lambda$ and $\sigma\in\widehat M(\tau)$. 
  Let $\lambda\in\R\setminus\{0\}$  and let $g\in G$ be fixed. Then for any $x\in G$ we have the following expansion  in $B^*(G,\tau)$
 $$ \Phi_{\sigma,\lambda}^\tau (g^{-1}x)v 
 \simeq  
 d_{\tau,\sigma}\tau^{-1}(k_2(x)) \sum_{s\in W} e^{(is\lambda-\rho)A^+(x)} 
 e^{(is\lambda-\rho)H(g^{-1}k_1(x))}\mathbf{c}(\tau,s\lambda) P_{s\sigma} \tau^{-1}(\kappa(g^{-1}k_1(x)))v 
 $$
 where $v\in V_\tau$ and $x=k_1(x) e^{A^+(x)}k_2(x)$.
 \end{proposition}
 
 \begin{proof}
  We prove the formula for $n$ odd. It follows from    
the estimate \eqref{main-formula1}, that 
\begin{align}\label{Phi11}
\Phi_\lambda^\tau(g^{-1}x)v\simeq d_{\tau,\sigma}\tau^{-1}(k_2(g^{-1}x))\sum_{s\in W} {\rm e}^{(is\lambda-\rho)A^+(g^{-1}x)} \mathbf{c}(\tau,s\lambda) P_{s\sigma}\tau^{-1}(k_1(g^{-1}x))v.
\end{align}
 Use \eqref{c-proj} and  write the right-hand side of \eqref{Phi11} in the form
\begin{align}
 \tau^{-1}(k_2(x))\sum_{s\in W}  \mathbf{c}(s\lambda)   {\rm e}^{(is\lambda-\rho){A^+(x)+H(g^{-1}k_1(x))}} P_{s\sigma}\tau^{-1}(\kappa(g^{-1}k_1(x))v +R_g(x),
\end{align}
where 
$$\begin{aligned} %\label{Phi21}
R_g(x) &=\tau^{-1}(k_2(g^{-1}x))\sum_{s\in W} {c}(s\lambda)  {\rm e}^{(is\lambda-\rho)A^+(g^{-1}x)} 
P_{s\sigma} \tau^{-1}(k_1(g^{-1}x))v\\
&-\tau^{-1}(k_2(x))\sum_{s\in W}  \mathbf{c}(s\lambda)  {\rm e}^{(is\lambda-\rho)(A^+(x)+H(g^{-1}k_1(x)))} 
 P_{s\sigma}\tau^{-1}(\kappa(g^{-1}k_1(x))v. 
\end{aligned}$$
So, we have to show that $R_g\simeq 0$. To do so, we need the following lemmas, which will be proved in the appendix (see Lemma \ref{cliffA} and Lemma \ref{key-fo11A}).  
\begin{lemma}\label{cliff}   For any $g, x\in G$ we have
\begin{align}\label{AH}
A^+(gx)=A^+(x)+H(gk_1(x))+E(g,x),
\end{align}
where the function $E$ satisfies the estimate
\begin{align}\label{AHapp}
0<E(g,x)\leq {\rm e}^{2(A^+(g)-A^+(x))}.
\end{align}
 \end{lemma}

\begin{lemma}\label{key-fo11}
  For any $g\in G$  we have
\begin{equation*}
\lim_{R\to\infty}   \tau^{-1}(k_2(g e^{RH_0})P_\sigma\tau^{-1}(k_1(g e^{RH_0})=P_\sigma\tau^{-1}(\kappa(g)) 
 \end{equation*}
  \end{lemma}
  Now if we let
  $$\begin{aligned}
  I_g(x)&=\tau^{-1}(k_2(g^{-1}x)) \sum_{s\in W} \mathbf{c}(s\lambda)                          {\rm e}^{(is\lambda-\rho)(A^+(x)+H(g^{-1}k_1(x))}\times \\
  &\hspace{5cm}\times [{\rm e}^{(is\lambda-\rho)E(g^{-1},x)}-1]
P_{s\sigma} \tau^{-1}( k_1(g^{-1}x)) v,
   \end{aligned}
  $$
  and
    $$\begin{aligned}
    J_g(x)=
    \sum_{s\in W} \mathbf{c}(s\lambda)                         & {\rm e}^{(is\lambda-\rho)(A^+(x)+H(g^{-1}k_1(x))}\times\\
   \hspace{3cm} &\times
[\tau^{-1}(k_2(g^{-1}x))P_{s\sigma} \tau^{-1}(k_1(g^{-1}x))-\tau^{-1}(k_2(x))P_{s\sigma}\tau^{-1}(\kappa(g^{-1}k_1(x)))]
  v
     \end{aligned}
  $$
  then using  
  the identity \eqref{AH},  
we can write $R_g(x)$ as  $R_g(x)=I_g(x)+J_g(x)$.
%\begin{align*}
%\begin{split}
%& R_g(x)=
%\underbrace{\tau^{-1}(k_2(g^{-1}x)) \sum_{s\in W} \mathbf{c}(s\lambda)                          {\rm e}^{(is\lambda-\rho)(A^+(x)+H(g^{-1}k_1(x))}[{\rm e}^{(is\lambda-\rho)E(g^{-1},x)}-1]
%P_{s\sigma} \tau^{-1}( k_1(g^{-1}x)) v}_{I_g(x)} \\
%&+
%\underbrace{\sum_{s\in W} \mathbf{c}(s\lambda)                          {\rm e}^{(is\lambda-\rho)(A^+(x)+H(g^{-1}k_1(x))}\times
%[\tau^{-1}(k_2(g^{-1}x))P_{s\sigma} \tau^{-1}(k_1(g^{-1}x))-\tau^{-1}(k_2(x))P_{s\sigma}\tau^{-1}(\kappa(g^{-1}k_1(x)))]
%  v}_{J_g(x)}.
%\end{split}
%\end{align*}
We will now  prove that $I_g\simeq 0$ and $J_g\simeq0$. We have 
\begin{align*}
&\frac{1}{R}\int_{B(R)}\diff {(xK)}\,\parallel I_g(x)\parallel^2  
\leq \\
&2\parallel v\parallel^2 \mid \mathbf{c}(\lambda)\mid^2\sum_{s\in W}
\frac{1}{R}\int_{B(R)}\diff {(xK)}\,{\rm e}^{-2\rho(A^+(x)+H(g^{-1}k_1(x))}\mid {\rm e}^{(is\lambda-\rho)E(g,x)}-1\mid^2 .
\end{align*}
 From \eqref{AHapp}  we have the estimate $|E(g,x)|\leq c_g {\rm e}^{-2A^+(x)}$ and therefore 
 $$ |{\rm e}^{(i\lambda-\rho)E(g,x)}-1|^2\leq (\lambda^2+\rho^2) c_g {\rm e}^{-2A^+(x)},$$ 
 thus,   $I_g\simeq 0$.
 
 It remains to prove that $J_g\simeq 0$.   To do so, we use  the Cartan decomposition $x=ka_th$ and the fact  that $(2\sinh t)^{2\rho} \leq {\rm e}^{2\rho t}$ to obtain
\begin{align*}
\begin{split}
& \frac{1}{R}\int_{B(R)}\diff x\parallel J_g(x)\parallel^2\leq
 2\sum_{s\in W} |\mathbf{c}(s\lambda)|^2 \times \\
& \frac{1}{R}
 \int_0^R \int_K {\rm e}^{-2\rho(H(g^{-1}k))}
\parallel \left( \tau^{-1}(k_2(g^{-1}ka_t)P_{s\sigma}\tau^{-1}(k_1(g^{-1}ka_t))-P_{s\sigma}\tau^{-1}(\kappa(g^{-1}k) \right)v\parallel^2 \diff x\diff t  \\
&=
 2\sum_{s\in W} |\mathbf{c}(s\lambda)|^2 \times \\
 &\int_0^1\int_K{\rm e}^{-2\rho(H(g^{-1}k))}
\parallel \left( \tau^{-1}(k_2(g^{-1}ka_{tR})P_{s\sigma}\tau^{-1}(k_1(g^{-1}ka_{tR}))-P_{s\sigma}\tau^{-1}(\kappa(g^{-1}k) \right)v\parallel^2 \diff x\diff t.\\
\end{split}
\end{align*}
But, by Lemma \ref{key-fo11} we have 
$$
\lim_{R\to \infty}
%\parallel \tau^{-1}(k_1(g^{-1}ka_{Rt})k_2(g^{-1}ka_{Rt}))-\tau^{-1}(\kappa(g^{-1}k))\parallel^2
\parallel \left( \tau^{-1}(k_2(g^{-1}ka_{tR})P_{s\sigma}\tau^{-1}(k_1(g^{-1}ka_{tR}))-P_{s\sigma}\tau^{-1}(\kappa(g^{-1}k) \right)v\parallel^2
=0.
$$
Since 
$${\rm e}^{-2\rho H(g^{-1}k)}
\parallel \left( \tau^{-1}(k_2(g^{-1}ka_{tR})P_{s\sigma}\tau^{-1}(k_1(g^{-1}ka_{tR}))-P_{s\sigma}\tau^{-1}(\kappa(g^{-1}k) \right)v\parallel^2
\leq 2 {\rm e}^{-2\rho(H(g^{-1}k))}\|v\|^2$$
 and noting that $\int_K {\rm e}^{-2\rho H(g^{-1}k)} {\rm d}k=1$ we   obtain by  Lebesgue's dominated convergence theorem that
\begin{align*}
\lim_{R\rightarrow \infty}\int_K{\rm e}^{-2\rho H(g^{-1}k)}
%\parallel \tau^{-1}(k_1(g^{-1}ka_{Rt})k_2(g^{-1}ka_{Rt}))-\tau^{-1}(\kappa(g^{-1}k))\parallel^2 
\parallel \left( \tau^{-1}(k_2(g^{-1}ka_{tR})P_{s\sigma}\tau^{-1}(k_1(g^{-1}ka_{tR}))-P_{s\sigma}\tau^{-1}(\kappa(g^{-1}k) \right)v\parallel^2
{\rm d}k=0.
\end{align*}
Thus $\lim_{R\to\infty}\frac{1}{R}\int_{B(R)}\parallel J_g(x)\parallel^2{\rm d}x=0$, and   the proposition follows.
 
\end{proof}

 For any $\lambda\in\R$, $g\in G$ and $v\in V_\tau$, we define $p_{\sigma,\lambda}^{g,v}\in L^2(K,\sigma)$  by 
  \begin{equation}\label{les-p}
   p_{\sigma,\lambda}^{g,v}(k)  =e^{(i\lambda-\rho)H(g^{-1}k)}P_\sigma \tau(\kappa(g^{-1}k))^{-1}v.
  \end{equation}

 % Let $w$ a representative in $M'/M$ of the nontrivial element of the Weyl group $W$, where $M'* is the normalizer of $A$ in $K$.
  
 The following lemma is easy to establish.
 \begin{lemma}\label{UU} $(1)$ For $\lambda\in\mathbb R\setminus\{0\}$, the set of finite combination of  $p_{\sigma,\lambda}^{g,v}$\,  $(v\in V_\tau$ and $g\in G)$ is a dense subspace of $L^2(K,\sigma)$.% (see eg. \cite{H2}).

$(2)$ For any $s\in W=\{Id,w\}\simeq\{1,-1\}$ and $\lambda\in\mathbb{R}\setminus\{0\}$, There exists a  unique unitary isomorphism  $U_{s,\lambda}: L^2(K,\sigma)\to L^2(K,s\sigma)$ such that   
  \begin{equation}\label{U-p}
  	  U_{s,\lambda} p_{\sigma,\lambda}^{g,v}=p_{s\sigma,s\lambda}^{g,v}
  \end{equation}
  Moreover, for $F_1 \in L^2(K,\sigma)$, $F_2 \in L^2(K,s\sigma)$ we have $\mathcal{P}_{\sigma,\lambda}^\tau F_1=\mathcal{P}_{s\sigma,s\lambda}^\tau F_2$ if and only if $U_{s,\lambda}F_1=F_2$ i.e. $U_{s,\lambda}=\left(\mathcal{P}_{s\sigma,s\lambda}^\tau\right)^{-1}\circ \mathcal{P}_{\sigma,\lambda}^\tau$.
  \end{lemma}

 The following theorem provides a scattering formula for Poisson integrals.
  
 \begin{theorem}\label{asymp-Poisson}
   Let $\tau\in \Lambda$ and let $\sigma\in\widehat M(\tau)$. For $\lambda\in\mathbb R\setminus\{0\}$ and  $F\in L^2(K,\sigma)$,   we have the following asymptotic expansions for the Poisson transform in $B^*(G,\tau)$ 
\begin{equation}\label{asympt30}
 \mathcal P_{\sigma,\lambda}^\tau F(x)\simeq d_{\tau,\sigma}\tau(k_2(g))^{-1} \sum_{s\in W} e^{(is\lambda-\rho)A^+(x)} \mathbf c(s\lambda,\tau) U_{s,\lambda} F(k_1(x)), 
  \end{equation}
 for any $x\in G$.  
  \end{theorem}
 \begin{proof}
   Notice that both side of \eqref{asympt30} depend continuously on $F\in L^2(K,\sigma)$. Therefore 
in view  of the item (1)   in Lemma \ref{UU}, to  prove    the asymptotic expansion holds in   $L^2(K,\sigma)$ it is sufficient to prove it for the functions    $F=p^{g,v}_{\sigma,\lambda}$.
Since
$$\mathcal P_{\sigma,\lambda}^\tau  p_{\sigma,\lambda}^{g,v}(x) =\Phi_{\sigma,\lambda}^\tau(g^{-1}x) v,$$
we get by Proposition \ref{key-formula}  
$$\mathcal P_{\sigma,\lambda}^\tau  p_{\sigma,\lambda}^{g,v}(x) \simeq  \tau(k_2(x))^{-1}\sum_{s\in W} c(s\lambda) e^{(is\lambda-\rho)A^+(x)}  p_{s\sigma,s\lambda}^{g,v}(k_1(x)),$$
then we use  the identity \eqref{U-p}  to conclude.
\end{proof}

 \section{ Proof  of  Theorem \ref{main-th-Poisson} -- Strichartz's conjecture for Poisson transform }
 %\begin{proof}[
 \noindent{\bf Proof  of item (1) - Theorem \ref{main-th-Poisson}.}
% ]
For $n$ even or odd, the right hand-side of the estimate \eqref{esti-Poisson16}    follows from Proposition \ref{pro-est-Poisson} and the left hand side form the equality \autoref{asym-Poisson16}.
Now, let us prove \autoref{asym-Poisson16} for $n$ odd.  
Let $F\in L^2(K,\sigma_{n-1}^\pm)$  and let $\varphi_F$ be the $V_{\tau_n}$-valued function defined by  %for $g=k_1g) e^{A^+(g)}k_2(g)$,
$$\varphi_F(g)= \tau_n(k_2(g))^{-1} 
\left( 
\mathbf{c}(\lambda) e^{(i\lambda-\rho)A^+(g)} F(k_1(g)) +  
 \mathbf{c}(-\lambda) e^{(-i\lambda-\rho)A^+(g)} U_{\omega,\lambda} F(k_1(g)) 
 \right).$$
Using change of variables $g=k_1e^{tH_0}k_2$ (see \autoref{chg-cartan}) and the fact that $\tau_n$ and $U_{\omega,\lambda}$ are unitary we have 
$$\begin{aligned}
\frac{1}{R} \int_{B(R)} \diff(gK)& \| \varphi_F(g)\|^2_{\tau_n}   
=2 |\mathbf{c}(\lambda)|^2 \| F\|^2_{L^2(K,\sigma_{n-1}^\pm)} \left(\frac{1}{R} \int_0^R  \diff t\,(2e^{-t}\sinh t)^{n-1} \right)\\
&+ 2 |\mathbf{c}(\lambda)|^2 \Re\left( \int_K \diff k\,\langle F(k),U_{-1,\lambda}F(k)\rangle_{\sigma_{n-1}^\pm} \frac{1}{R}\int_0^R \diff t\,e^{2(i\lambda-\rho)t} (2\sinh t)^{n-1}\right),
\end{aligned}
$$
thus, by taking the limit and observing that $2\mid \mathbf{c}(\lambda)\mid^2=\gamma_0 \nu(\lambda)^{-1}$,  it becomes clair that %we easily see that  
\begin{equation}\label{lim}
\lim_{R\to\infty} \frac{1}{R}  \int_{B(R)} \diff(gK) \| \varphi_F(g)\|^2_{\tau_n} =\gamma_0 \nu(\lambda)^{-1} \|F\|^2_{L^2K,\sigma^\pm)}.
\end{equation}
Furthermore, 
 $$\begin{aligned}
 \frac{1}{R}\int_{B(R)} \diff(gK) \|\mathcal P_{\sigma_{n-1}^\pm}^{\tau_n} F(g) \|^2_{\tau_n}  &=
 \frac{1}{R}\int_{B(R)}\diff(gK)  \Bigl( \| \varphi_F(g)\|^2_{\tau_n} +\| \mathcal P_{\sigma_{n-1}^\pm}^{\tau_n} F(g)  -\varphi_F(g)\|^2_{\tau_n} \\
& +  2\Re \langle \mathcal P_{\sigma_{n-1}^\pm}^{\tau_n} F(g)  -\varphi_f(g), \varphi_F(g)\rangle_{\tau_n} \Bigr),
 \end{aligned}
 $$
hence,   \autoref{asym-Poisson16} follows from \autoref{lim}, Theorem \ref{asymp-Poisson} and the Cauchy-Schwarz inequality.
The proof when $n$ is even can be done in the same way. \\
%\end{proof}

It remains to show that $\mathcal P_{\sigma,\lambda}^\tau$ is an isomorphism. Before, we need to establish some intermediate results.
Let $\tau\in\Lambda$ and $\sigma\in\widehat{M}(\tau)$ 
   of dimension $d_{\sigma}$.  Let $\widehat K(\sigma)\subset \widehat K$ be the subset of unitary equivalence classes of irreducible representations containing    $\sigma$ upon restriction to $M$. Consider an element $(\delta,V_\delta)$   in $\widehat{K}(\sigma).$
   From    \cite{BSilva}  or \cite{IT},  it   follows     that $\sigma$ occurs in $\delta_{|M}$ with multiplicity one, and therefore $\dim \mathrm{Hom}_M(V_\delta, V_\sigma)=1$.  Choose  the orthogonal projection $P_\delta : V_\delta\to V_\sigma$ to be   a generator of $\mathrm{Hom}_M(V_\delta, V_\sigma)$.
 Fix an orthonormal basis    $\{v_j: j=1,\ldots, d_\delta=\dim V_\delta\}$ of $V_\delta$. 
 Then  the set of functions 
$\{  \phi^\delta_j: 1\leqslant j\leqslant d_\delta, \; \delta\in \widehat{K}(\sigma)\}$ defined by 
$$k\mapsto \phi^\delta_j(k)=P_\delta(\delta(k^{-1})v_j) $$ 
is  an orthogonal basis of the space $L^2( K, \sigma )$, see, e.g., \cite{wallach}. 
 Hence, the  Fourier series expansion of each   $ F$ in $L^2( K, \sigma)$ is given by   
$$F(k)=\sum_{\delta\in\widehat{K}(\sigma)}\sum_{j=1}^{d_\delta} a^\delta_{j} \phi^{\delta}_j(k),$$
with 
\begin{equation*}%\label{L2-norm}
\displaystyle \Vert F\Vert^2_{L^2(K,\, \sigma)}=\sum_{\delta\in\widehat{K}(\sigma)} \frac{d_\sigma}{d_\delta} \sum_{j=1}^{d_\delta}\mid a^\delta_{j}\mid^2.   
\end{equation*}

For $\lambda\in\mathbb R\setminus\{0\}$ and   $(\delta,V_\delta)\in\widehat K(\sigma)$, define the following  Eisenstein integrals $\Phi_{\lambda,\delta}  \colon G\to \operatorname{End}(V_\delta,V_\tau)$ by
\begin{equation}\label{Eisen}
  \Phi_{\lambda,\delta}(g)(v)=  \sqrt{d_{\tau,\sigma}} \int_K \diff k\,{\rm e}^{-(i\lambda+\rho)H(g^{-1}k)}\tau(\kappa(g^{-1}k))   P_\delta(\delta(k^{-1}) v).
\end{equation}

We will now prove  that the Poisson transform is a surjective map from $L^2(K,\sigma)$ onto $\mathcal E^2_{\sigma,\lambda}(G,\tau)$. 
 
 Let $f\in \mathcal E_{\sigma,\lambda}^2(G,\tau)$.
Since $\lambda\in\mathbb R\setminus\{0\}$, then by \cite[Theorem 4.16]{olbrich}, there exists $F\in C^{-\omega}(K,\sigma)$ such that $\mathcal P_{\sigma,\lambda}^\tau F=f$.
Let  $F(k)=\sum_{\delta\in\widehat K(\sigma )}  \sum_{j=1}^{d_\delta} a^\delta_{j}P_\delta (\delta(k^{-1}))v_j$   be its Fourier expansion. Observing that  $\Phi_{\lambda,\delta}(g)v_j=\mathcal{P}_{\sigma,\lambda} (P_\delta \delta()^{-1}v_j)(g)$ we can easily see that
 $$
f(g)=\sum_{\delta\in \widehat{K}(\sigma)} {\sum_{j=1}^{d_\delta}}a^\delta_j \Phi_{\lambda,\delta}(g)v_j
$$
in $C^\infty(G,\tau)$. By the Schur  orthogonality  relations we have 
$$
\frac{1}{R}\int_{B(R)}\diff (gK)\parallel f(g)\parallel^2
=\sum_{\delta\in \widehat{K}(\sigma)}{\sum_{j=1}^{d_\delta}}\mid a^\delta_j\mid^2 \frac{1}{R}\int_{B(R)}\diff (gK)\parallel \Phi_{\lambda,\delta}(g)v_j\parallel^2,
$$
hence
$$\sum_{\delta\in \widehat{K}(\sigma)}{\sum_{j=1}^{d_\delta}}\mid a^\delta_j\mid^2 \frac{1}{R}\int_{B(R)}\diff (gK)\parallel \Phi_{\lambda,\delta}(g)v_j\parallel^2<\parallel f\parallel^2_\ast<\infty,
$$
for any $R>0$.

In the above inequality  the summation over $\delta\in \widehat K(\sigma)$   remains true in particular over $\delta\in \Delta$ for any finite subset $\Delta\subset \widehat K(\sigma)$. 
As
$$
\lim_{R\to \infty}\frac{1}{R}\int_{B(R)}\diff (gK)\parallel \Phi_{\lambda,\delta}(g)v_j\parallel^2=\gamma_0\nu(\lambda)^{-1}\frac{d_\sigma}{d_\delta},
$$
we get 
$$
\|f\|_*^2\geq   \gamma_0 \nu(\lambda)^{-1} \sum_{\delta\in \Delta}\sum_{j=0}^{d_\delta} \frac{d_\sigma}{d_\delta} |a_j^\delta|^2,
$$
and since $\Delta$ is arbitrary, we obtain
$$\gamma_0 \nu(\lambda)^{-1} 
\sum_{\delta\in\widehat K(\sigma)}  \sum_{j=1}^{d_\delta}  \frac{d_{\sigma} }{d_\delta}|a_j^\delta|^2\leq \|f\|_*^2<\infty,$$
which proves that $F\in L^2(K,\sigma)$ and completes the proof of the theorem.  \\

%\end{proof}

    \section{Proof  of  Theorem \ref{main-th-proj} - Strichartz's conjecture for spectral projections}  
     We will present the proof for the case where  $n$  is odd. The argument for even  $n$  follows in a similar manner.
 % We will give the proof for $n$ odd. The case when $n$ is even can be proved in the same way.

%\begin{proof}[
\noindent{\bf Proof of   item (1) -- Theorem \ref{main-th-proj}.}
%]
 Let $f\in L^2(G,\tau_n)$. Since
 $\mathcal{Q}_{\sigma, \lambda}^\tau f=   {\nu(\lambda)}   \mathcal{P}_{\sigma, \lambda}^\tau\left(\mathcal{F}_{\sigma, \lambda}^\tau f\right)$
 for a.e.  $\lambda \in \mathbb{R}\setminus\{0\}$, the uniform estimate \autoref{esti-Poisson-odd} of the Poisson transform    implies
$$\begin{aligned}
 \sup_{R>0}\frac{1}{R} \int_{B(R)}\diff(gK)\left\|\mathcal{Q}_{\lambda}^{\tau_n} f(g)\right\|_{\tau_n}^2   
 &= \sup_{R>0}\frac{1}{R} \int_{B(R)}\diff(gK)\left\|\mathcal{Q}_{\sigma_{n-1}^+,\lambda}^{\tau_n} f(g)\right\|_{\tau_n}^2   \\
 &\qquad \qquad \qquad  +\sup_{R>0}\frac{1}{R} \int_{B(R)}\diff(gK)\left\|\mathcal{Q}_{ \sigma_{n-1}^-, \lambda}^{\tau_n} f(g)\right\|_{\tau_n}^2   \\
  &\leq C^2  \nu(\lambda) \| \mathcal F_{\sigma_{n-1}^+,\lambda}^{\tau_n}f   \|^2_{L^2(K,\sigma_{n-1}^+)} %\\
  %&\qquad \qquad \qquad
  + C^2   \nu(\lambda) \| \mathcal F_{\sigma_{n-1}^-,\lambda}^{\tau_n}f    \|^2_{L^2(K,\sigma_{n-1}^-)},  \\
\end{aligned}
$$ 
and therefore from   the Plancherel  formula \autoref{Plan-Formula-odd}   we get
$$\begin{aligned}
 \| \mathcal Q^{\tau_n} f\|_+^2=\sup_{R>0} \int_0^\infty \diff\lambda \frac{1}{R} \int_{B(R)}\diff(gK)\left\|\mathcal{Q}_{\lambda}^{\tau_n} f(g)\right\|_{\tau_n}^2  
  &\leq  C^2    \| f \|^2_{L^2(G,\tau_n)},
\end{aligned}
$$ 
which proves the right hand side of \autoref{esti-main1}.  
Further, by \autoref{asym-Poisson16} and  \autoref{esti-Poisson16}  we have successively 
$$\begin{aligned}
 \lim_{R\to\infty}\frac{1}{R} \int_{B(R)} \diff(gK)\left\|\mathcal{Q}_{\lambda}^{\tau_n} f(g)\right\|_{\tau_n}^2  
  &=\gamma_0\nu(\lambda)    \| \mathcal F_{\sigma_{n-1}^+,\lambda}^{\tau_n}f   \|^2_{L^2(K,\sigma_{n-1}^+)} + \gamma_0\nu(\lambda)  \| \mathcal F_{\sigma_{n-1}^-,\lambda}^{\tau_n}f    \|^2_{L^2(K,\sigma_{n-1}^-)},
\end{aligned}$$
and  
$$ 
\begin{aligned}
\frac{1}{R} \int_{B(R)}\diff(gK)\left\|\mathcal{Q}_{\lambda}^{\tau_n} f(g)\right\|_{\tau_n}^2  &\leq 
C^2  \nu(\lambda) \|\mathcal F_{\sigma_{n-1}^+,\lambda}^{\tau_n}f   \|^2_{L^2(K,\sigma_{n-1}^+)} + C^2   \nu(\lambda) \| \mathcal F_{\sigma_{n-1}^-,\lambda}^{\tau_n}f    \|^2_{L^2(K,\sigma_{n-1}^-)}.	
\end{aligned}
$$
Then by the Lebesgue dominated convergence and the Placherel formula \autoref{Plan-Formula-odd} we get
$$ \begin{aligned}
\lim_{R\to\infty}\frac{1}{R}  \int_0^\infty \diff\lambda\int_{B(R)}\diff(gK)\left\|\mathcal{Q}_{\lambda}^{\tau_n} f(g)\right\|_{\tau_n}^2 
&=\int_0^\infty \diff\lambda \Bigl(\gamma_0 \nu(\lambda)  \| \mathcal F_{\sigma_{n-1}^+,\lambda}^{\tau_n} f   \|^2_{L^2(K,\sigma_{n-1}^+)}  \\
&  + \gamma_0 \nu(\lambda) \| \mathcal F_{\sigma_{n-1}^-,\lambda}^{\tau_n} f   \|^2_{L^2(K,\sigma_{n-1}^-)}\Bigr)  \\
&=\gamma_0\|f\|^2_{L^2(G,\tau_n)},
\end{aligned}$$
which prove \autoref{asym-Q} and  indeed the left hand side of \autoref{esti-main1}. \\
%\end{proof}

%\begin{proof}[
\noindent{\bf Proof of item (2) -- Theorem \ref{main-th-proj}.}
%]
 Linearity together with  \autoref{esti-main1} imply the injectivity of  $\mathcal Q^{\tau_n}$. It remains then to prove that $ \mathcal Q^{\tau_n}$ is onto. 
 Let $\psi \in\mathcal E^2_{\mathbb R_+}(G,\tau_n)$. Then for any $\lambda$, $\psi (\lambda,\cdot) =\psi _{+,\lambda}+\psi _{-,\lambda}$ with $\psi _{\pm,\lambda}\in 
 \mathcal E_{\sigma_{n-1}^\pm,\mathbb R_+}^2(G,\tau_n)$.
Since
$$\sup_{R>0} \frac{1}{R}\int_{B(R)} \diff(gK) \|\psi_{\pm,\lambda} (g)\|_{\tau_n}^2  <\infty \quad \text{a.e.}\; \lambda\in (0,\infty),$$
then applying    Theorem \ref{main-th-Poisson}  we can assert the existence of $F_{\pm,\lambda}\in L^2(K,\sigma_{n-1}^\pm)$  such that 
for a.e. $\lambda$ we have $\psi_{\pm,\lambda}=\nu(\lambda)\mathcal P^{\tau_n}_{\sigma_{n-1}^\pm,\lambda} F_{\pm,\lambda}$.
Furthermore, it follows from the estimate \eqref{esti-Poisson16} that
$$\begin{aligned}
\sup_{R>0}\frac{1}{R}  \int_{B(R)} \diff(gK)\|\psi_{\pm,\lambda} (g)\|_{\tau_n}^2   
&\geq
C^{-2} \nu(\lambda)\int_K \diff k\,\|F_{\pm,\lambda}(k)\|^2_{\sigma_{n-1}^\pm}, 
\end{aligned}
$$
hence
$$ 
C^{-2}\int_0^\infty   \int_K \diff k\|F_{\pm,\lambda}(k)\|^2_{\sigma_{n-1}^\pm} \nu(\lambda)\diff\lambda \leq \|\psi_{\pm,\lambda} \|_+^2 <\infty,$$
 which proves that 
 $$F_\lambda=F_{+,\lambda}+F_{-,\lambda} \in L^2(\R_+,\nu(\lambda)d\lambda ;L^2(K,\sigma_{n-1}^+))\oplus
  L^2(\R_+,\nu(\lambda)d\lambda ;L^2(K,\sigma_{n-1}^-)).$$ 
  Now, by the Plancherel theorem, there exists $f\in L^2(G,\tau_n)$ such that 
  $$F_\lambda=\mathcal F^\tau f(\lambda)=(\mathcal F_{\sigma_{n-1}^+,\lambda}^{\tau_n} f,\mathcal F_{\sigma_{n-1}^-,\lambda}^{\tau_n} f) ,$$ 
  thus, for any $g\in G$,
  $$\begin{aligned}
\psi(\lambda,g) &=\nu(\lambda) \mathcal P_{\sigma_{n-1}^+,\lambda}^{\tau_n} \left( \mathcal F_{\sigma_{n-1}^+,\lambda}^{\tau_n} f \right)(g) +
 \nu(\lambda) \mathcal P_{\sigma_{n-1}^-,\lambda}^{\tau_n} \left( \mathcal F_{\sigma_{n-1}^-,\lambda}^{\tau_n} f \right)(g)\\
 & =\mathcal Q_\lambda^{\tau_n}( f)(g)=\mathcal Q^{\tau_n} f(\lambda,g).
 \end{aligned}$$
 This completes the proof of Theorem \ref{main-th-proj}.
 %\end{proof}

  \section{Appendix}
   In this appendix we will provide the proof the lemmas \ref{key-fo11} and Lemma \ref{cliff}, we have used to show the key formula in Proposition \ref{key-formula}.

There is a useful matrix realization of  the conformal spin group $\Spin_0(1,n)$    which 
 mimics the use of $\mathrm{SU}(1,1)$ or $\mathrm{SL}(2,\mathbb R)$. This realization makes possible to see the group $\Spin_0(1,n)$ acting either conformally on $S^{n-1}\simeq \partial H^n(\mathbb R)$ or as the double cover of the group of orientation preserving Möbius transformations of the unit ball $B^n\simeq H^n(\mathbb R)$.
  This approach has been  initiated by Vahlen \cite{Vahlen}  and 
brought to the fore independently  by Ahlfors \cite{Ahlfors}   and  Takahashi \cite{takahashi}.

Here are some definitions and properties for Vahlen matrices.
We begin by recalling basic knowledge related to Clifford algebras (see e.g. \cite{del-al}, \cite{LM}).  Let $(e_1,\cdots, e_{n})$ be the natural basis of $\R^{n}$. The Clifford algebra $\cl_{0,n}= \cl_n=\cl(\R^{n})$ will be  the associative algebra (with unit 1) generated by $\R^{n}$ subject to the relations 
$$xy+yx=-2\langle x,y\rangle\quad (x, y\in\R^{n}).$$
In particular
$$e_i^2=-1 \; \text{ and } \;  e_ie_j=-e_je_i \; \text{ for } i\not=j.$$

  Recall the three involutions of the Clifford algebra $\cl_{0,n}$: 

 -- the Clifford conjugaison defined on vectors by $\bar x=-x$ and satisfies $\overline{ab} = \bar b\bar a$;
 
 -- the main involution defined on vectors by $x'=-x$ and satisfies $(ab)' = a' b'$;
 
 -- the reversion defined on vectors by $x^*=x$ and satisfies $(ab)^* = b^* a^*$; which  is in fact the composition of the Clifford conjugaison and main involution: $a^*=\bar a'=\overline{a'}$.

There exists a canonical inner product on $\cl_{0,n}$ extending the inner product on $\R^{n}$ and we will denote the  corresponding norm   by $|\cdot |$.

The Clifford group $\Gamma_{0,n}$ is the set of all elements of the Clifford algebra $\cl_{0,n}$ which can be written as products of non-zero vectors in $\cl_{0,n}^1$.

 As was observed by Vahlen, the  conformal spin group $\Spin_0(1,n)\subset \cl_{1, n} \simeq \cl_{0, n-1} \otimes \cl_{1,1} \simeq \cl_{0, n-1} \otimes \mathcal{M}_2(\mathbb{R}) \simeq\mathcal{M}_2\left(\mathcal{C} \ell_{0, n-1}\right)$ 
 can be  realized as the group
 $$\underline G=\left\{\begin{pmatrix} a & b\\ b' &a'\end{pmatrix} \; :\; a, b\in \Gamma_{0,n-1}\cup\{0\}, \; b a^*\in \R^{1,n-1} \; \text{ and }\; |a|^2 -|b|^2=1 %|a|^2-|b|^2=1
  \right\}.$$
 The set 
 $$\underline K=\left\{\begin{pmatrix} u&0\\ 0&u' \end{pmatrix}\; :\; u\in \Gamma_{0,n-1}, \text{ and }\, |u|^2=1\right\},$$ 
 is  a maximal compact subgroup of $\underline G$, isomorphic to $\Spin(n)$.  We use the identification
\begin{equation*}
\begin{pmatrix} u&0\\ 0&u' \end{pmatrix}\in \underline K \leftrightarrow  u\in \Gamma_{0,n-1}^0
\end{equation*}
where $\Gamma_{0,n-1}^0$ is the subset of elements in $\Gamma_{0,n-1}$ with norm equal 1.

The group $\underline G$ has the Iwasawa and  Cartan decompositions 
$$\underline G=\underline K\,\underline A\,\underline N, \quad G=\underline K\{a_t; t\geq 0\}\underline K$$ where $a_t=\begin{pmatrix}
\cosh t&\sinh t\\\sinh t&\cosh t
\end{pmatrix}$,
$\underline A=\{ a_t; t\in \mathbb{R}\}$ and $\underline N$ is a subgroup of $\underline G$ consisting of matrices of the form  $n_x=\begin{pmatrix}
1-x&x\\-x&1+x
\end{pmatrix}$ with $x\in\mathbb{R}^{n-1}$.

Let $g=\begin{pmatrix}
a&b\\a'&b'
\end{pmatrix}\in \underline G$ and write $g=\kappa(g){\rm e}^{H(g)}n$ with respect to the Iwasawa decomposition. Then we have 
\begin{align}\label{Iwa}
H(g)=\log \mid a+b\mid,
\end{align}
and
\begin{align}\label{Iwa1}
 \kappa(g)=\begin{pmatrix}
\frac{a+b}{\mid a+b\mid}&0\\0&\frac{a'+b'}{\mid a+b\mid}
\end{pmatrix}.
\end{align} 
Also, throughout the Cartan decomposition we may  write $g=k_1(g){\rm e}^{A^+(g)}k_2(g)$ with
\begin{align}\label{Cart}
A^+(g)=\log (\mid a\mid+\mid b\mid),
\end{align}
and 
\begin{align}\label{Cart1}
k_1(g)k_2(g)=\begin{pmatrix}
\frac{a}{\mid a\mid}&0\\0&\frac{a'}{\mid a\mid}
\end{pmatrix}.
\end{align}

Now we shall prove the two lemmas  using the above realization. For the convenience of the reader, we will recall their statement.

\begin{lemmab}\label{cliffA}   For any $g, x\in G$ we have
$$
A^+(gx)=A^+(x)+H(gk_1(x))+E(g,x),
$$
where the function $E$ satisfies the estimate
$$
0<E(g,x)\leq {\rm e}^{2(A^+(g)-A^+(x)}.
$$
\end{lemmab}

\begin{proof}
Put $E(g,x)=A^+(gx)-A^+(x)-H(gk_1(x))$ and write  $x=k{\rm e}^{tH_0}h$ with respect to the Cartan decomposition. Then
$$
E(g,x)=A^+(gk{\rm e}^{tH_0})-H(gk)-t.
$$
  Since  \eqref{Iwa} and \eqref{Cart}  imply   $H(y)< A^+(y)$ for any $y\in G$, we  get $E(g,k{\rm e}^{tH_0})> 0$ .
 
Next, let $g=\begin{pmatrix}
a&b\\b'&a'
\end{pmatrix}$ and $k=\begin{pmatrix}
u&0\\0&u'
\end{pmatrix}$. Then
\begin{align}
E(g,x)=\log[\mid \cosh t\,au+\sinh t\,bu'\mid+\mid \sinh t\, au+\cosh t\,bu'\mid]-\log\mid au+bu'\mid-t,
\end{align}
and a direct calculation shows that 
\begin{align*}
E(g,x)\leq \log \left(1+{\rm e}^{-2t}\, \frac{\mid a\mid+\mid b\mid}{\mid a\mid-\mid b\mid}\right).
\end{align*}
Therefore the upper estimate   in the lemma follows from the above inequality and the fact that $\displaystyle\frac{\mid a\mid+\mid b\mid}{\mid a\mid-\mid b\mid}={\rm e}^{2\,A^+(g)}$.
\end{proof}

\begin{lemmab}\label{key-fo11A} Let $\tau\in \Lambda$ and $\sigma\in\widehat{M}(\tau)$.
  For any $g\in G$  we have
\begin{equation*}%\label{lm1}
\lim_{R\to\infty}   \tau^{-1}(k_2(g e^{RH_0})P_\sigma\tau^{-1}(k_1(g e^{RH_0})=P_\sigma\tau^{-1}(\kappa(g)) 
 \end{equation*}
  \end{lemmab}

   \begin{proof}
If $n$ is even, then $P_\sigma$ is the identity map and we have to prove the following identity
\begin{equation}\label{key-fo1}
\lim_{R\to\infty} \tau(k_2(ge^{RH_0})k_1(ge^{RH_0}))=\tau(\kappa(g)).
  \end{equation}
Write $g=\begin{pmatrix}
a&b\\
b'&a'	
\end{pmatrix}
$. In view of \eqref{Cart1} we have 
\begin{align*}
k_1(g{\rm e}^{RH_0})k_2(g{\rm e}^{RH_0})=
\begin{pmatrix}
\frac{a\cosh R+b\sinh R }{\mid a\cosh R +b\sinh R\mid} & 0\\
0&\frac{a'\cosh R+b'\sinh R}{\mid a'\cosh R+b'\sinh R\mid}
\end{pmatrix},
\end{align*}
from which we deduce \eqref{key-fo1} and the lemma follows for $n$ even.

 Suppose $n$ is odd and consider the case of $\sigma=\sigma_{n-1}^+$. Then for any $v\in V_{\tau_n}$ and any $k\in K$
 $$ \begin{aligned}
 P_{\sigma_{n-1}^+}\tau_n(k)v
 &= 
 \frac{1}{2}[
 \tau_n(k)v+\gamma(\tau_n(k)v)
 ]
 \\
&= 
 \frac{1}{2}[
 \tau_n(k)v+\tau_n(\alpha(k))\gamma(v)
 ]
  &= 
 \frac{1}{2}[
 \tau_n(k)v+\tau_n(k)\gamma(v)
 ].
  \end{aligned}
  $$
  Here $\gamma$ is the map  inducing the grading   $V_{\tau_n}=V_{\sigma_{n-1}^+}\oplus V_{\sigma^-_{n-1}}$ and $\alpha$ the main involution of the Clifford algebra.
  Applying this in our situation we get
  $$
  \begin{aligned}
 & \|
  \tau^{-1}_n(k_2(ge^{RH_0})  )P_{\sigma_{n-1}^+}\tau^{-1}_n(k_1(ge^{RH_0}))v - 
  P_{\sigma_{n-1}^+}\tau_n^{-1}(\kappa(g))v
  \|\\
  &  \leq \frac{1}{2}
   \|
    \tau^{-1}_n(k_2(ge^{RH_0}) )  \tau_n^{-1}(k_1(ge^{RH_0}))v  -\tau_n^{-1}(\kappa(g))v
   \|\\
   &+\frac{1}{2}
   \|
   \tau^{-1}_n(k_2(ge^{RH_0}) )\tau_n^{-1}(k_1(ge^{RH_0}))\gamma(v) -\tau_n^{-1}(\kappa(g))\gamma(v)
   \|.
  \end{aligned}$$
By \eqref{key-fo1}, the last two terms tend to 0 whenever $R\to\infty$. Thus 
$$
\lim_{R\to \infty} \|
  \tau^{-1}_n(k_2(ge^{RH_0})  )P_{\sigma_{n-1}^+}\tau^{-1}_n(k_1(ge^{RH_0}))v - 
  P_{\sigma_{n-1}^+}\tau_n^{-1}(\kappa(g))v
  \|\\=0.
$$
 The  proof in other case ($\sigma=\sigma_{n-1}^-$) is similar, then it is omitted.
\end{proof}

  %____________
 \pdfbookmark[1]{References}{ref}


\begin{thebibliography}{99}
\footnotesize\itemsep=0pt
\providecommand{\eprint}[2][]{\href{http://arxiv.org/abs/#2}{arXiv:#2}}
 
\bibitem{Ahlfors} Ahlfors L. Möbius transformations in $\mathbb{R}^n$ expressed through $2 \times 2$ matrices of Clifford numbers, 
\href{https://doi.org/10.1080/17476938608814142}{Complex variables}  {\bf 5} (1986) 215--224.
 
 
 \bibitem {Anker} Anker J.-Ph. 
 A basic inequality for scattering theory on Riemannian symmetric spaces of the noncompact type.
\href{https://doi.org/10.2307/2374832}{Amer. J. Math.} {\bf 113} (1991), no. 3, 391--398. 
 
 
\bibitem{BSilva} Baldoni-Silva M.W.  
Branching theorems for semisimple Lie groups of real rank one.
\href{http://www.numdam.org/item/RSMUP_1979__61__229_0/}
{\it Rend. Sem. Mat. Univ. Padova} {\bf 61} (1979), 229--250.



\bibitem{BBK2} Bensaïd S.; Boussejra A.; Koufany K. On Poisson transforms for spinors. 
\href{ https://doi.org/10.2140/tunis.2023.5.771 }{Tunis. J. Math.} {\bf 5} (2023) no.4, 771-792. 


\bibitem{BBK} Bensaïd S.; Boussejra A.; Koufany K. On Poisson transforms for differential forms on real hyperbolic spaces. 
\href{https://arxiv.org/abs/2112.09994}
{arXiv:2112.09994}.

   
  \bibitem{BK-24}  Boussejra A.; Koufany K. A characterization of the L2-range of the generalized spectral projections related to the Hodge-de Rham Laplacian. 
\href{https://arxiv.org/abs/2406.00536}
{ arXiv:2406.00536}.


\bibitem{BIO} Boussejra A.; Imesmad N.; Ouald Chaib A. 
  $L^2$-Poisson integral representations of eigensections of invariant differential operators on a homogeneous line bundle over the complex Grassmann manifold $S U(r, r+b) / S(U(r) \times U(r+b))$.
 \href{https://doi.org/10.1007/s10455-021-09819-9}{Ann. Global Anal. Geom.} {\bf 61} (2022), no. 2, 399--426.
 
 
\bibitem{BO} Boussejra A., Ouald Chaib A.
A characterization of the $L^2$-range of the Poisson transforms on a class of vector bundles over the quaternionic hyperbolic spaces.
\href{https://doi.org/10.1016/j.geomphys.2023.105019}{J. Geom. Phys.} {\bf 194} (2023), Paper No. 105019, 24 pp.

 \bibitem{BS} Boussejra A. Sami H.
Characterization of the ${L}^p$-Range of the Poisson Transform in Hyperbolic Spaces ${\bf B}(\mathbb F^n)$. 
\href{https://www.emis.de/journals/JLT/vol.12_no.1/1.html}{J. Lie Theory} {\bf 12} (2002)  no. 1, 1--14.
  
  \bibitem{IB} Boussejra A.; Intissar A.
  Caractérisation des intégrales de Poisson-Szegö de $L^2(\partial B^n)$ dans la boule de Bergman $B^n$ , $n\geq 2$.  
  \href{https://gallica.bnf.fr/ark:/12148/bpt6k5470663g/f625.image.r=boussejra?rk=42918;4}{C. R. Acad. Sci. Paris Sér. I Math.} {\bf 315} (1992), no.13, 1353--1357.


\bibitem{BOS}   Branson Thomas P.; \'Olafsson, G.; Schlichtkrull, H. A Bundle valued Radon transform, with applications to invariant wave equations. \href{https://doi.org/10.1093/qmath/45.4.429}{Quart. J. Math. Oxford} (2) {\bf 45}, (1994), 429--461.

\bibitem{Obray} Bray William O.
 Aspects of harmonic analysis on real hyperbolic space, in: \href{https://doi.org/10.1201/9781003072133-5}{Fourier Analysis, Orono, ME, 1992,}  Lect. Notes Pure Appl. Math., vol. 157, Dekker, New York, 1994, pp. 77--102.
 
\bibitem{Campo} Camporesi. R, The Helgason Fourier transform for homogeneous vector bundles over Riemannian
symmetric spaces.
\href{https://doi.org/10.2140/pjm.1997.179.263}{Pac. J. Math.} {\bf 179} (2) (1997) 263--300.

\bibitem{CP} Camporesi, R.; Pedon, E.
Harmonic analysis for spinors on real hyperbolic spaces.  
\href{https://doi.org/10.4064/cm87-2-10 }{\it Colloq. Math.} {\bf 87} (2001), no. 2, 245--286.

%\bibitem{CH} R. Camporesi and A. Higuchi, On the eigenfunctions of the Dirac operator on spheres and real hyperbolic spaces, J. Geom. Phys. 20 (1996), 1--18.

%\bibitem{Camp} Camporesi, R. Harmonic analysis for spinor fields in complex hyperbolic spaces. {\it Adv. Math.} {\bf 154} (2000), no. 2, 367--442.

\bibitem{del-al} Delanghe, R.; Sommen, F.; Souček, V.
Clifford algebra and spinor-valued functions.
  Mathematics and its Applications, {\bf 53}. \href{https://doi.org/10.1007/978-94-011-2922-0}{\it Kluwer Academic Publishers Group, Dordrecht}, 1992. xviii+485 pp

\bibitem{FJ}Flensted-Jensen, M. Paley-Wiener type theorems for a differential operator connected with symmetric spaces
 Ark. Mat. {\bf 10}(1972), 143--162.


\bibitem{Gaillard} P.-Y. Gaillard, Harmonic spinors on hyperbolic space. \href{https://doi.org/10.4153/CMB-1993-037-7}{Canad. Math. Bull.} 36 (1993), 257--262.

\bibitem{H1} Helgason S. A duality for symmetric spaces with applications to group representations. \href{https://doi.org/10.1016/0001-8708(70)90037-X}{Adv. Math.} {\bf 5} (1970) 1--154.

\bibitem{H} S. Helgason, Groups and Geometric  Analysis, Integral geometry, Invariant Differential operators and Spherical Functions.  Academic Press, New York, 1984.

%\bibitem{H2} S. Helgason, Geometric Analysis on Symmetric Spaces, second edition, Math. Surveys Monogr., vol.39, American Mathematical Society, Providence, RI, 2008.

%\bibitem{GM} Gilbert, John E.; Murray, Margaret A. M.
%Clifford algebras and Dirac operators in harmonic analysis.
%Cambridge Studies in Advanced Mathematics, {\bf 26}. {\it Cambridge University Press, Cambridge}, 1991. viii+334 pp

\bibitem{GW} Goodman, R.; Wallach, Nolan R. Symmetry, representations, and invariants.  \href{https://doi.org/10.1007/978-0-387-79852-3}{Graduate Texts in Mathematics}, 255.
 {\it Springer, Dordrecht,} 2009. xx+716 pp

%\bibitem{IT} {\color{red} Ikeda Taniguchi }
  \bibitem{IT} Ikeda A., Taniguchi Y., Spectra and eigenforms of the Laplacian on $S^n$ and $P^n(\mathbb C)$. 
\href{https://dlisv03.media.osaka-cu.ac.jp/contents/osakacu/sugaku/111F0000002-01503-4.pdf}{\it Osaka J. Math.} {\bf 15} (1978), 515--546.

%\bibitem{Jost} Jost J., 
%Riemannian Geometry and Geometric Analysis, Fourth edition.  \href{https://doi-org.ezproxy.math.cnrs.fr/10.1007/978-3-642-21298-7}{Universitext}. {\it Springer-Verlag, Berlin,}2005. xiv+566

\bibitem{Ionescu} Ionescu Alexandru D.
On the Poisson Transform on Symmetric Spaces of Real Rank One.
\href{https://doi.org/10.1006/jfan.2000.3590}{J. Funct. Anal.} {\bf 174} (2000), no.2, 513–-523. 

\bibitem{Kaizuka} Kaizuka K.
A characterization of the L2-range of the Poisson transform related to Strichartz conjecture on symmetric spaces of noncompact type
 \href{https://doi.org/10.1016/j.aim.2016.08.020}{Adv. Math.} {\bf 303} (2016), 464–-501.

%\bibitem{Pratyoosh-Kumar-2014} Kumar, P.
%Fourier restriction theorem and characterization of weak $L^2$
%eigenfunctions of the Laplace–Beltrami operator.
%\href{https://doi.org/10.1016/j.jfa.2013.10.009}{J. Funct. Anal.}  {\bf 266} (2014), no. 9, 5584--5597.


%\bibitem{Kumar-Ray-Sarkar-2010} Kumar, P.; Ray, Swagato K.; Sarkar, Rudra P. The role of restriction theorems in harmonic analysis on harmonic $NA$ groups  
%\href{https://doi.org/10.1016/j.jfa.2010.01.001}{J. Funct. Anal.} {\bf 258} (2010), no. 7, 2453--2482.


 \bibitem{Ko}  Koornwinder, Tom H.,
Jacobi functions and analysis on non-compact semisimple Lie groups. 
\href{https://doi.org/10.1007/978-94-010-9787-1_1}
{\it Special functions: group theoretical aspects and applications}, 1--85,
Math. Appl., {\it Reidel, Dordrecht}, 1984.


 

\bibitem{LM}  Lawson, H. Blaine, Jr.; Michelsohn, Marie-Louise, Spin geometry. \href{https://doi.org/10.1515/9781400883912}{Princeton Mathematical Series}, {\bf 38}. {\it Princeton University Press, Princeton, NJ}, 1989. xii+427 pp

%\bibitem{Lohoue-Rychener-84}
 % Lohoué N.; Rychener T. Some function spaces on symmetric spaces related to convolution operators. 
 % \href{https://doi.org/10.1016/0022-1236(84)90010-7}{J. Funct. Anal.} {\bf 55} (1984) 200--219.

 


\bibitem{olbrich} Olbrich, M. Die Poisson-Transformation für homogene Vektorbiindel, Ph.D. Thesis, Humboldt  Universitlt, Berlin, 1994. 


% \bibitem{Pedon} Pedon, E.,  Analyse harmonique des formes différentielles sur l'espace hyperbolique réel. Thèse Université Henri Poincaré-Nancy (1997).


 
 
 
 
 
%\bibitem{Stein1} Stein, E.M., Interpolation of linear operators. 
%\href{Trans. Am. Math. Soc.} {\bf 83} (1956) 482--492.

%\bibitem{Stein1}  Stein, E.M.,  Harmonic Analysis: Real-Variable Methods, Orthogonality, and Oscillatory Integrals, 
%\href{https://www.jstor.org/stable/j.ctt1bpmb3s}{Princeton Math. Ser.} vol. 43, Princeton University Press, Princeton, NJ, 1993.

  

% \bibitem{strichartz1} Robert S. Strichartz, Local harmonic analysis on spheres, \href{https://doi.org/10.1016/0022-1236(88)90095-X}{J. Funct. Anal.} {\bf 77} (1988), no. 2, 403--438.
 
\bibitem{strichartz}  Robert S. Strichartz, Harmonic analysis as spectral  theory of Laplacians. 
\href{https://doi.org/10.1016/0022-1236(89)90004-9}{J. Funct. Anal.} {\bf 87} (1989) 51--148.
 
%\bibitem{strichartz3}  Robert S. Strichartz, $L^P$ harmonic analysis and Radon transforms on the Heisenbery group, 
%\href{https://doi.org/10.1016/0022-1236(91)90066-E}{J. Funct. Anal.} {\bf 96} (1991), no.2, 350--406.

 

%\bibitem{olbrich} Olbrich, M. Die Poisson-Transformation für homogene Vektorbiindel, Ph.D. Thesis, Humboldt  Universitlt, Berlin, 1994. 

\bibitem{takahashi} Takahashi, R.,
Série discrète pour les groupes de Lorentz $\mathrm{SO}_0(2n,1)$.  Colloque sur les Fonctions Sphériques et la Théorie des Groupes (Univ. Nancy, Nancy, 1971), Exp. No. 7, 10 pp. Inst. Élie Cartan, Univ. Nancy, Nancy, 1971. 

%\bibitem{Tomas} Tomas, P.  A restriction theorem for the Fourier transform. Bull. Am. Math. Soc. 81 (1975) 477--478.

 

\bibitem{Vahlen}  Vahlen K., Über Bewegungen und komplexe Zahlen, Math. 
\href{https://zenodo.org/record/2517224/files/article.pdf}{Annalen} {\bf 55} (1902) 585--593.


 \bibitem{wallach} Wallach, N. R. Harmonic analysis on homogenous spaces.
 \href{https://link.springer.com/chapter/10.1007/978-3-662-09756-4_1}{Pure and Applied Mathematics}, No. {\bf 19}. {\it Marcel Dekker, Inc., New York}, 1973. xv+361 pp.
 
%  \bibitem{Yang} Yang, A. Poisson transform on vector bundles.  \href{https://doi.org/10.1090/S0002-9947-98-01659-6 }  {\it  Trans. Amer. Math. Soc.}, {\bf 350} (3), 857-887 (1998).
 


\end{thebibliography}
\end{document}